%% file: main.tex
\documentclass[10pt,twoside]{amsart}
\usepackage{mathtools,amsmath,amssymb}
\usepackage{mathrsfs}
\usepackage{eucal}

\usepackage{libertine}       
\usepackage[a4paper,top=1in, bottom=1in, left=1in, right=1in]{geometry} 
\usepackage{graphicx} 
\usepackage[colorlinks=true,linkcolor=blue,citecolor=black]{hyperref} 

\usepackage{cleveref}
\usepackage[backend=biber,backref=true,doi=false,url=false,maxnames=9]{biblatex} 
\renewbibmacro*{pageref}{%
  \iflistundef{pageref}
    {}
    {\addspace
     \mkbibbrackets{\printlist[pageref][-\value{listtotal}]{pageref}}}} 
\usepackage{lipsum}
\usepackage{adjustbox} 
\usepackage{tikz-cd}
\usepackage{mathptmx}
\usepackage{quiver}
\usepackage{float}
\usepackage[english]{babel}
\usepackage{amsthm}
\usepackage{mathtools}
\usepackage{soul}
\usepackage{nicematrix}
\usepackage{array}   
\usepackage{booktabs}
\usepackage{float}
\usepackage{orcidlink}
\usepackage{csquotes}
\usepackage{booktabs}
\usepackage{geometry}
\usepackage{multirow}
\usepackage{nccmath}
\usepackage[normalem]{ulem}

\usepackage{nomencl}

\makenomenclature
\usepackage[toc,page]{appendix}
\usepackage{caption}
\newcommand{\GL}{\mathrm{GL}}
\newcommand{\SL}{\mathrm{SL}}
\newcommand{\tr}{\mathrm{tr}}

\newcommand{\M}{\mathrm{M}}
\newcommand{\im}{\mathrm{Im}}

\newcommand{\R}{\mathcal{O}_2}

\usepackage{thmtools, thm-restate}
\newtheorem{theorem}{Theorem}[section]
\newtheorem{lemma}[theorem]{Lemma}
\newtheorem{proposition}[theorem]{Proposition}
\newtheorem{corollary}[theorem]{Corollary}
\theoremstyle{definition} 
\newtheorem{definition}[theorem]{Definition}

\newtheorem{remark}[theorem]{Remark}
\newtheorem{question}{Question}

\title[Engel word on $\mathfrak{sl}_2^{\circ}(\mathcal{O})$ ]{ Surjectivity of Engel Maps over trace zero matrices in  $\M_2(\mathcal{O})$ }
\date{\today}
\author[Roy]{Ayon Roy}
\email[(Roy)]{ayonroy1999@gmail.com}
\address{Indian Institute of Science Education and Research Pune, Dr. Homi Bhabha Road, Pashan, Pune 411 008, India}
\author[Singh]{Anupam Singh}
\email[(Singh)]{anupamk18@gmail.com}
\address{Indian Institute of Science Education and Research Pune, Dr. Homi Bhabha Road, Pashan, Pune 411 008, India}

\date{\today}
\thanks{Roy is supported by an IISER Pune PhD Fellowship. Singh is funded by an ANRF-MATRICS Grant ANRF/ARGM/2025/000095/MTR}
\begin{document}
\begin{abstract}
The surjectivity of various noncommutative polynomials has been studied extensively on Lie algebras over fields of different characteristics. In this article, we study the surjectivity of Engel Maps over trace zero matrices in $\M_2(\mathcal{O})$, where $\mathcal{O}$ is a local principal ideal ring complete with respect to its maximal ideal and has a residue field $k$ of characteristic $\neq 2$. We show that the image of $(m+1)$-th Engel map induced by the Engel polynomial $e_{m+1}(x, y) = [\cdots[[x, \underbrace{y], y], \dots, y]}_{m+1 \text{ times}}$ over $\M_2(\mathcal{O})$ can be determined by the image of the corresponding Engel map over $\M_2(k)$, for $m\geq 1$. Moreover, we prove that the $(m+1)$-th Engel map on $\mathfrak{sl}_2^{\circ}(\mathcal{O})=\left\{\ A\in \M_2(\mathcal{O})\ \mid\ \tr(A)=0\right\}$ is surjective if and only if the corresponding Engel map on $\mathfrak{sl}_2(k)$ is surjective. Under some mild condition on the residue field $k$, our results shows that every $g\in \mathfrak{sl}_2^{\circ}(\mathcal{O})$ can be expressed as $g = e_{m+1}(h_1, h_2)$ for $m\geq 1$, where both $h_1, h_2$ are in $\mathfrak{sl}_2^{\circ}(\mathcal{O})$.
\end{abstract}

\maketitle
\input{intro}
\input{Lie-Engel-maps}

\input{fibers}
\input{surj}
\printbibliography
\vspace{2em}
\end{document}

%% file: intro.tex
\section{Introduction}

Let $\mathcal{R}$ be a commutative ring with unity and $A\in \M_n(\mathcal{R})$. If there exist matrices $P, Q\in \M_n(\mathcal{R})$ such that $A$
can be written as $A=[P, Q] = PQ - QP$ then the trace of $A$ must be zero. A natural question is whether every trace-zero matrix in $\M_n(\mathcal{R})$ can be written as a single commutator. This was first proved for matrices over fields of characteristic zero by Shoda~\cite{shodacommutator} and later by Albert and Muckenhoupt~\cite{albertcommutator} for any general field. Over time, this question has been investigated for $\M_n(\mathcal{R})$ for more general rings $\mathcal{R}$, including principal ideal domains and the ring of integers. The first result for domains other than fields was proved by Lissner \cite{lissnercommutator}, who showed that every $2 \times 2$ trace-zero matrix over a PID is a commutator. This work was motivated by its connection to a special case of Serre’s problem on projective modules, later known as the Quillen–Suslin theorem; see \cite[Sections 1-2] {lissnercommutator}. Lissner’s result was later proved in full generality by Stasinski \cite{stasinskicommutator, Stasinski2}, who showed that every trace-zero matrix over a principal ideal domain (PID) is a commutator, thereby answering Laffey’s \cite{laffey1997lectures} question of whether every trace-zero matrix over a Euclidean domain is a commutator. Moreover, Stasinski \cite[Corollary 6.4]{stasinskicommutator} and \cite[Theorem 5.3]{Stasinski2} showed that if $\mathcal{R}$ is a principal ideal ring (e.g. $\mathbb{Z}/p^{\ell}\mathbb{Z}$ ) then any trace zero matrix in $\M_n(\mathcal{R})$ is also a commutator. This result used the fact that every principal ideal ring (PIR) is a finite product of rings, each of which is a homomorphic image of a PID; as established by Hungerford; see \cite [Lemma 10 and Corollary 11]{hungerford}.

A natural question that arises in this context is whether every trace-zero matrix $A\in \M_2(k)$ can be written as a commutator $A=[P, Q]$, where both $P$ and $Q$ are also trace-zero matrices in $\M_2(k)$, for any arbitrary field $k$. When $k$ has characteristic other than $2$, Thompson \cite[Theorem 3 (ii)]{thompsonij} proved that this is indeed the case. Combining this along with the result due the Lissner discussed above, we can observe that, for any odd prime $p$, any lift of a trace zero matrix of $\M_2(\mathbb{F}_p)$ in $\M_2(\mathbb{Z}_p)$, which has zero trace value, can be written as a single commutator $[\Lambda_1, \Lambda_2]$ where $\Lambda_1, \Lambda_2 \in \M_2(\mathbb{Z}_p)$ and $\mathbb{Z}_p$ denotes the ring of $p$-adic integers. However, Lissner's result does not guarantee that $\Lambda_1$ and $\Lambda_2$ also have trace zero. The proof techniques used in \cite[Theorem 6.3]{stasinskicommutator} and \cite[Theorem 5.3]{Stasinski2} ensure that it is possible to choose $\Lambda_1$ and $\Lambda_2$ as a trace-zero matrix in $\M_2(\mathbb{Z}_p)$. This motivates us to study the lifting problems in the context of Engel maps over trace-zero matrices in $\M_2(\mathcal{O})$, where $\mathcal{O}$ is a complete local ring with a residue field of characteristic other than two. 

For a commutative ring $\mathcal{R}$ with unity $\mathcal{R}\left\langle X, Y\right\rangle$ denotes the free associative algebra generated by two variables
$X, Y$ over $\mathcal{R}$. The elements are \emph{non-commutative polynomials} in two variables. For a
non-commutative polynomial $f\in \mathcal{R} \left\langle X, Y\right \rangle$ and a $\mathcal{R}$-algebra $\mathcal{S}$, we denote the images set of $f$  in $\mathcal{S}$ by $f(\mathcal{S}) = \left\{f(a,b) \mid a, b\in \mathcal{S}\right\}$. The $(m+1)$-th \emph{Engel polynomials} are defined by 
$$e_{m+1}(X, Y) = [\cdots[[X, \underbrace{Y], Y], \dots, Y]}_{m+1 \text{ times}}$$ 
which is in $ \mathcal{R}\left\langle X,Y\right\rangle$. Here $e_1(X, Y) = [X, Y] := XY - YX$ and $e_{m+1}(X, Y) = [e_m(X, Y), Y]$ for $m\geq 1$. Consider the matrix algebra $\mathcal{S} = \M_2(\mathcal{R})$. The $(m+1)$-th \emph{Engel map} (or Engel polynomial map) $e_{m+1}^{\mathcal{R}}\colon \M_2(\mathcal{R}) \times \M_2(\mathcal{R})\rightarrow \M_2(\mathcal{R})$ given by $(g_1, g_2) \mapsto e_{m+1}(g_1, g_2)$ defined by evaluation. In other words 
\begin{eqnarray*}
e_1^{\mathcal{R}}(g_1, g_2) &=& [g_1, g_2] = g_1g_2-g_2g_1\\ 
e_{m+1}^{\mathcal{R}}(g_1, g_2) &=& [e_m^{\mathcal{R}}(g_1, g_2), g_2] = e_m^{\mathcal{R}}(g_1, g_2)g_2-g_2e_m^{\mathcal{R}}(g_1, g_2)\ \text{for}\ m\geq 1.
\end{eqnarray*}
For a large enough field $K$ with characteristic other than two, Bandman et. al. (see \cite[Corollary 4.4]{bandmanliealgebra}) proved that $e_{m+1}^{K}\colon \mathfrak{sl}_2(K)\times \mathfrak{sl}_2(K)\rightarrow \mathfrak{sl}_2(K)$, defined by evaluation as above, is surjective. In the case $\mathcal{R}$ is a field we simply denote $e_{m+1}^{\mathcal{R}}$ by $\bar e_{m+1}$. The preceding discussions lead us to the following questions:
\begin{question}{\label{qn1}}
Let $\mathcal{O}$ be a local principal ideal ring, complete with respect to its maximal ideal $\mathfrak{m}=(\pi)$ with its residue field $k$ of characteristic $\neq 2$. Let the $(m+1)$-th Engel map $\bar e_{m+1}\colon \M_2(k) \times \M_2(k)\rightarrow \M_2(k)$ be such that $\im(\bar e_{m+1})=\mathfrak{sl}_2(k)$ for $m\geq 1$. Is it so that $\im(e_{m+1}^{\mathcal{O}}) = \mathfrak{sl}_2^{\circ}(\mathcal{O})$; where $e_{m+1}^{\mathcal{O}}\colon \M_2(\mathcal{O})\times \M_2(\mathcal{O})\rightarrow \M_2(\mathcal{O})$ and $\mathfrak{sl}_2^{\circ}(\mathcal{O}) = \left\{A\in\M_2(\mathcal{O})\mid\ \tr(A)=0\right\}$?
\end{question} 
\noindent For clarification we remark that $\mathfrak{sl}_2(\mathcal{O}) = \left\{A\in\M_2(\mathcal{O})\mid\ \tr(A) \in \mathfrak{m} \right\}$ and here we deal with $\mathfrak{sl}_2^{\circ}(\mathcal{O})$.
Further, if the answer to the question~\ref{qn1} is affirmative, then we ask the analogous question restricted to $\mathfrak{sl}_2^{\circ}(\mathcal{O})$. 
\begin{question}{\label{qn2}}
Let $\mathcal{O}$ be a local principal ideal ring, complete with respect to its maximal ideal $\mathfrak{m} = (\pi)$ with its residue field $k$ of characteristic $\neq 2$. Suppose the $(m+1)$-th Engel map $\bar e_{m+1}\colon \mathfrak{sl}_2(k)\times \mathfrak{sl}_2(k)\rightarrow \mathfrak{sl}_2(k)$ defined by evaluation is surjective. Does it imply that the $(m+1)$-th Engel map $ e_{m+1}^{\mathcal{O}}\colon \mathfrak{sl}_2^{\circ}(\mathcal{O}) \times \mathfrak{sl}_2^{\circ}(\mathcal{O}) \rightarrow \mathfrak{sl}_2^{\circ}(\mathcal{O})$ defined by evaluation is also surjective for $m\geq 1$?
\end{question}

In this article, we provide an affirmative answer to both questions. Theorem~\ref{th1} and Theorem~\ref{th2} answers  Question~\ref{qn1} and Question~\ref{qn2} respectively. Using Theorem~\ref{th2} together with the result due to Bandman et. al. \cite[Corollary 4.4]{bandmanliealgebra} we have been able to conclude that for large enough residue field $k$ with charcteristic $\neq 2$, the map $ e_{m+1}^{\mathcal{O}} \colon \mathfrak{sl}_2^{\circ}(\mathcal{O})\times \mathfrak{sl}_2^{\circ}(\mathcal{O}) \rightarrow \mathfrak{sl}_2^{\circ}(\mathcal{O})$ is surjective for $m\geq 1$; see Corollary~\ref{largeenough}. Our methods are motivated by the lifting strategy used in \cite{AvniGelanderKassabovShalev} and \cite{RoySingh2026}. Moreover, the study of fibers of an element with respect to the $(m+1)$-th Engel map over the residue field level plays a central role in proving the main theorems.

\subsection{Notation and Conventions}
Now, we recall some definitions and set the notation for the rest of the article. A commutative ring $\mathcal{R}$ with unity is said to be \emph{complete with respect to an ideal $I$} if the canonical map $\mathcal{R} \rightarrow \varprojlim\limits_{j\geq 1} \mathcal{R}/I^j\mathcal{R}$ is an isomorphism. We denote by $\mathcal{O}$ a local principal ideal ring that is complete with respect to its unique maximal ideal $\mathfrak{m} = (\pi)$. We assume that the residue field $k$ has characteristic $\neq 2$. For $\ell\geq 1$, we denote the quotient ring $\mathcal{O}_{\ell} = \mathcal{O}/ \pi^{\ell}\mathcal{O}$. This is a local principal ideal ring of length $\ell$ with unique maximal ideal $\mathfrak{m}_\ell = (\pi_\ell)$, where $\pi_\ell = \pi + (\pi^\ell)$. The canonical map $\mathcal{O}_{\ell + 1} \rightarrow \mathcal{O}_{\ell}$ is denoted by $\theta_\ell$. The kernel of the natural surjection $\theta_j \colon \mathcal{O}_{j+1} \rightarrow \mathcal{O}_j$ is $ker(\theta_j) = \frac{\mathfrak{m}^j}{\mathfrak{m}^{j+1}}$. Therefore, $ker(\theta_j)^2=0$ for each $j\geq 1$. Note that $\mathcal{O}_1 = \mathcal{O}/\mathfrak{m} \cong k$ and we simply denote $\theta_1 = \theta$.

For a local ring $\mathcal{A}$ with residue field $k$ and its unique maximal ideal $\mathfrak{m}$, let $\M_2(\mathcal{A})$ denote the set of all $2\times 2$ matrices with entries from $\mathcal{A}$. The general linear group $\GL_2(\mathcal{A})$ is the set of all elements $X\in \M_2(\mathcal{A})$ such that $\det(X) \in \mathcal{A}^{\times}$. The special linear group $\SL_2(\mathcal{A})$ is the set of all elements $X\in \M_2(\mathcal{A})$ with $\det(X) = 1$. The quotient map $\theta \colon \mathcal{A} \rightarrow k =\mathcal{A}/ \mathfrak{m}$ induces a canonical map $\theta \colon \M_2(\mathcal{A}) \rightarrow \M_2(k)$ denoted by $X\mapsto \overline{X}$. Thus, we get a map $\theta\colon \GL_2(\mathcal{A}) \rightarrow \GL_2(k)$ and $\theta \colon \SL_2(\mathcal{A}) \rightarrow \SL_2(k)$. Further, note that the quotient map $\theta\colon \mathcal{A} \rightarrow k$ also induces a surjective map on the polynomial ring $\theta \colon \mathcal{A}[x_1, x_2, \ldots, x_n] \rightarrow k[x_1, x_2, \ldots, x_n]$ defined by reduction of the coefficients under $\theta$. By abuse of notation, all the above reduction maps are denoted simply by $\theta$, which is usually clear from the context. Similarly, all the reductions from $\mathcal{O}_{\ell+1}$ level to $\mathcal{O}_{\ell}$ level are denoted by $\theta_{\ell}$. For any matrix $A\in \M_2(\mathcal{A})$ its characteristic polynomial is denoted by $\chi_A(t)$ which is $\det(tI - A)$. It is well known that, for any commutative ring with unity, every square matrix over that ring satisfies its own characteristic equation. Therefore $\chi_A(A) = \mathbf{0}$. A minimal degree monic annihilating polynomial of $A$ is said to be a minimal polynomial of $A$. 

The commutator word map $e_1^{\mathcal{A}} \colon \M_2(\mathcal{A}) \times \M_2(\mathcal{A}) \rightarrow \M_2(\mathcal{A})$ given by $(X, Y)\mapsto XY - YX = [X, Y]$ induces a commutator word map at residue field level $\overline{e}_1 \colon \M_2(k) \times \M_2(k) \rightarrow \M_2(k)$. For $m\geq 1$, the induced $(m+1)$-th Engel word map $e_{m+1}^{\mathcal{A}} \colon \M_2(\mathcal{A}) \times \M_2(\mathcal{A}) \rightarrow \M_2(\mathcal{A})$ is defined recursively as $e_{m+1}^{\mathcal{A}}(X,Y) = e_1^{\mathcal{A}}(e_{m}^{\mathcal{A}}(X,Y), Y)$. We denote $e_0^{\mathcal{A}}(X, Y) = X$. The reduced Engel word map over the field level is denoted by $\Bar{e}_{m+1}\colon \M_2(k) \times \M_2(k)\rightarrow \M_2(k) $ which is defined in a similar way (thus $\Bar{e}_{m+1}$ is $e_{m+1}^k$). The notation $\im(e_{m+1}^{\mathcal{A}})$ denotes the image of the $(m+1)$-th Engel word $e_{m+1}^{\mathcal{A}}$ over the additive group $\M_2(\mathcal{A})$. Therefore, $\im(\bar e_{m+1})$ means the image of the $(m+1)$-th Engel word $\bar e_{m+1}$ over the group $\M_2(k)$. Throughout this article, we assume that the characteristic of the residue field $k$ is $\neq 2$.

%% file: Lie-Engel-maps.tex
\section{Engel maps over Lie algebra: a lifting perspective}{\label{englifting}}

The Lie algebra of $\SL_n(k)$ is denoted by $\mathfrak{sl}_n(k) = \{P \in \M_n(k) \mid \tr(P) = 0\}$. We have a nondegenerate symmetric bilinear form given by $\langle P, Q \rangle = \tr(PQ)$ over the $k$ vector space $\mathfrak{sl}_n(k)$, when $char(k)$ is odd and $char(k) \nmid n$. With respect to this bilinear form, the orthogonal complement of a subspace $\mathfrak{U} \subset \mathfrak{sl}_n(k)$ will be denoted by $\mathfrak{U}^{\perp}$. Moreover, for an element $D\in \M_n(k)$, we use the notation $Z(D)$ to denote the centralizer of $D$ in $\mathfrak{sl}_n(k)$, i.e., $Z(D) = Z_{\M_n(k)}(D) \cap \mathfrak{sl}_n(k)$; where $Z_{\M_n(k)}(D)=\left\{P\in\M_n(k) \mid\ PD=DP\right\}$. For any finite dimensional vector space $\mathcal{V}$ if there is a non degenerate symmetric bilinear form $\mathfrak{B}$ associated with it, then for any subspace $\mathcal{W}$ of $\mathcal{V}$ we have $dim(\mathcal{V}) = dim(\mathcal{W}) + dim(\mathcal{W}^{\perp})$, where $\mathcal{W}^{\perp}$ is the orthogonal complement of $\mathcal{W}$ with respect to the bilinear form $\mathfrak{B}$. Moreover, this dimension formula and $\mathcal{W} \subset (\mathcal{W}^{\perp})^{\perp}$ together implies $(\mathcal{W}^{\perp})^{\perp} = \mathcal{W}$, in this case.

Let us recall the Lie algebra $\mathfrak{sl}_2(\mathcal{O}_{j+1}) = \{P\in \M_2(\mathcal{O}_{j+1}) \mid \tr(P)\in \mathfrak{m}_{j+1} \}$ for $j\geq 1$. Note that $\mathfrak{sl}_2(\mathcal{O}_{j+1})$ is generated by $\left\{ \left(\begin{array}{cc} 1 & 0 \\  0 & -1 \end{array}\right), \left(\begin{array}{cc}  0 & 1 \\ 0 & 0
\end{array}\right), \left(\begin{array}{cc} 0 & 0 \\ 1 & 0 \end{array}\right), \left(\begin{array}{cc} 0 & 0 \\ 0 & \pi_{j+1}
\end{array}\right)\right\}$ as a $\mathcal{O}_{j+1}$ module and hence it is a finitely generated $\mathcal{O}_{j+1}$ module for each $j \geq 1$. As each $\mathcal{O}_{j+1}$ is a principal ideal ring, it is Noetherian. Hence, any $\mathcal{O}_{j+1}$-submodule of the finitely generated $\mathcal{O}_{j+1}$-module $\mathfrak{sl}_2(\mathcal{O}_{j+1})$ is again finitely generated by \cite[Proposition 6.2 and 6.5]{Atiyahmacdonald}. Let us recall the notion of regular semisimple and cyclic elements.
\begin{definition}
Let $\mathcal{A}$ be a local ring with residue field $k$. An element $g\in \M_n(\mathcal{A})$ is said to be regular semisimple if $\bar g$ is a regular semisimple matrix in $\M_n(k)$. In other words, the characteristic polynomial of $\bar g$ has distinct roots in $\bar k$ (the algebraic closure of $k$). 
\end{definition}
\begin{definition}
Let $\mathcal{A}$ be a local ring with residue field $k$. An element $g\in \M_n(\mathcal{A})$ is said to be cyclic if $\bar g$ is a cyclic matrix in $\M_n(k)$. In other words, $k^n$ is a cyclic $k[t]$-module with the action of $t$ given via $\bar g$. 
\end{definition}
Consider the canonical isomorphism $\mathcal{T^*} \colon  \M_n(\mathcal{O}) \rightarrow \varprojlim\limits _{j\geq 1} \M_n(\mathcal{O}_j)$ defined by $A\mapsto (A_1, A_2, \ldots )$ where $A_j=[a_{u,v}  + \mathfrak{m}^j]$ for $j\geq 1$ (here $a_{u,v}$ is the $(u,v)$-th entry of the matrix $A=[a_{u,v}]$). Thus, at times, we think of an element $A\in \SL_n(\mathcal{O})$, as an element of $\M_n(\mathcal{O})$ and visualize it as $(A_1, A_2, \ldots)$ under the above isomorphism $\mathcal{T}^*$ (see also \cite[section 2]{PRS2026}). Note that $A_{j+1}$ is a lift of $A_j$ for each $j$, thus the determinant of each $A_j$ is $1$ and $A_j\in \SL_n(\mathcal{O}_j)$ for all $j\geq 1$. The following Lemma relates the lifting of a solution to each $j$-th level ring $\mathcal O_{j}$ to that over $\mathcal O$. 
\begin{lemma}{\label{lifting complete level}}
Let $\mathcal{O}$ be a local principal ideal ring, complete with respect to its maximal ideal $\mathfrak{m}=(\pi)$, and it has a residue field $k$ of characteristic $\neq 2$. Let $m\geq 1$ be a given positive integer and $A=(A_1, A_2, \ldots) \in \M_n(\mathcal{O})$. Suppose $A_1 = \bar e_{m+1}(X_1, Y_1)$ for some $X_1, Y_1\in \M_n(k)$. If for each $A_{j+1}\in \M_n(\mathcal{O}_{j+1})$ there exists lift  $(X_{j+1}, Y_{j+1})\in \M_n(\mathcal{O}_{j+1})^2$ of $(X_j, Y_j)$ with $A_{j+1} = e_{m+1}^{\mathcal{O}_{j+1}}(X_{j+1}, Y_{j+1})$ where $A_j = e_{m+1}^{\mathcal{O}_j}(X_j, Y_j)$ then, $A\in  \im(e_{m+1}^{\mathcal{O}})$.  
\end{lemma}
\begin{proof}
We have, for each $A_{j+1}\in \M_n(\mathcal{O}_{j+1})$ there exists lift  $(X_{j+1},Y_{j+1})\in \M_n(\mathcal{O}_{j+1})^2$ of $(X_j,Y_j)$ (i.e. $X_{j+1}\equiv X_j\ (\text {mod}\ ker(\theta_j))$) such that $A_{j+1}=e_{m+1}^{\mathcal{O}_{j+1}}(X_{j+1}, Y_{j+1})$ where $A_j = e_{m+1}^{\mathcal{O}_j}(X_j, Y_j)$. Consider $\widehat{X}=(X_1, X_2,\ldots)$ and $\widehat{Y} = (Y_1, Y_2,\ldots)$ in $\varprojlim\limits _{j\geq 1} \M_n(\mathcal{O}_j)$. Let us denote $X=\mathcal{T}^{*-1}(\widehat{X})$ and $Y = \mathcal{T}^{*-1}(\widehat{Y})$. Then, from the definition of $\mathcal{T}^*$ it is easy to see that $\mathcal{T}^*(e_{m+1}^{\mathcal{O}}(X, Y))=(\bar e_{m+1}(X_1, Y_1), e_{m+1}^{\R}(X_2, Y_2),\ldots) = (A_1, A_2,\ldots)$. As $\mathcal{T}^*$ is an isomorphism therefore $A = e_{m+1}^{\mathcal{O}}(X, Y)$.
\end{proof}
\noindent If $A\in\M_n(\mathcal{O})$ is a cyclic matrix, then in the above notation $A_j\in \M_n(\mathcal{O}_j)$ must be cyclic for all $j$, and they have the same minimal and characteristic polynomial for each $j\geq 1$; see \cite[corollary 2.3]{PRS2026}. We explain some notation before proceeding further. 

Consider the Engel map $\bar e_{m+1} \colon \M_2(k)^2 \rightarrow \M_2(k)$ induced by the natural reduction of coefficients. For $\mathbf{\bar g}_1, \mathbf{\bar g}_2 \in \M_2(k)$ we have $\bar e_{m+1}(\mathbf{\bar g}_1,\mathbf{\bar g}_2)\in \M_2(k)$. Take any lift $\mathbf{g}$ of $\bar e_{m+1}(\mathbf{\bar g}_1, \mathbf{\bar g}_2)$ in $\M_2(\mathcal{O}_2)$ such that $\tr(\mathbf{g}) = 0$ (i.e. $\mathbf{g} \in \mathfrak{sl}_2^{\circ}(\mathcal{O}_2)$). In other words, one can consider the lift with respect to the natural surjection $\mathfrak{sl}_2^{\circ}(\mathcal{O}_2)\rightarrow\mathfrak{sl}_2(k)$. Then, there exists $\mathbf{g}_1, \mathbf{g}_2$ (lifts of $\mathbf{\bar g}_1, \mathbf{\bar g}_2$ respectively) in $\M_2(\mathcal{O}_2)$ such that $\mathbf{g} = e_{m+1}^{\R}(\mathbf{g}_1, \mathbf{g}_2)+ \pi_2 \mathcal{D}$, for some $\mathcal D\in \M_2(\mathcal{O}_2)$. This is because the kernel of the surjection $\M_2(\mathcal{O}_{{\ell}+1}) \rightarrow \M_2(k)$ is $\M_2(\mathfrak{m}_{\ell + 1}^{\ell})$ for $\ell\geq 1$. As $\tr(\mathbf{g}) = 0 =\tr(e_{m+1}^{\R}(\mathbf{g}_1, \mathbf{g}_2))$, therefore $\tr(\mathcal{D})\in \mathfrak{m}_2$, where $\pi_2$ is the generator of $\mathfrak{m}_2$ with $\pi_2^2 = 0$. Therefore $\mathcal{D} \in \mathfrak{sl}_2(\mathcal{O}_2)$. Finally, the Lie algebra $\mathfrak{sl}_2(\mathcal{O}_2) = \{P \in \M_2(\mathcal{O}_2) \mid \tr(P) \in ker(\mathcal{O}_2 \rightarrow K) = \mathfrak{m}_2\}$.

\subsection{Derivative map corresponding to $e_{m+1}$ :} 
Now we look at the expression of the derivative map corresponding to the $(m+1)$-th Engel word. Let $g_1, g_2\in \M_2(\mathcal{O}_{j+1})$ and $X, Y \in \mathfrak{sl}_2(\mathcal{O}_{j+1})$. We start with the expression of $e_1^{\mathcal{O}_{j+1}}(g_1 + \delta_jX, g_2 + \delta_jY)$ and associated derivative map $De_1^{\mathcal{O}_{j+1}}$ and inductively build the expression for $De_{m+1}^{\mathcal{O}_{j+1}}$ where $\delta_j = \pi_{j+1}^j$, i.e., $\delta_j^2=0$.

\begin{eqnarray*}
e_1^{\mathcal{O}_{j+1}}(g_1+\delta_jX,g_2+\delta_jY) 
&=& (g_1+\delta_jX)(g_2+\delta_jY)-(g_2+\delta_jY)(g_1+\delta_jX)\\
&=& (g_1g_2-g_2g_1)+\delta_j(g_1Y+Xg_2-g_2X-Yg_1)\\
&=& e_1^{\mathcal{O}_{j+1}}(g_1,g_2)+\delta_j\left( e_1^{\mathcal{O}_{j+1}}(X,g_2)-e_1^{\mathcal{O}_{j+1}}(Y,g_1)\right)
= e_1^{\mathcal{O}_{j+1}}(g_1,g_2)+\delta_j De_1^{\mathcal{O}_{j+1}}(X,Y)
 \end{eqnarray*}
where $De_1^{\mathcal{O}_{j+1}} \colon \mathfrak{sl}_2(\mathcal{O}_{j+1}) \times \mathfrak{sl}_2(\mathcal{O}_{j+1}) \rightarrow \mathfrak{sl}_2(\mathcal{O}_{j+1})$ is defined by \begin{equation}{\label{der 1}}
    (X,Y)\mapsto e_1^{\mathcal{O}_{j+1}}(X,g_2)-e_1^{\mathcal{O}_{j+1}}(Y,g_1).
\end{equation}
Now,\begin{eqnarray*}
&& e_2^{\mathcal{O}_{j+1}}(g_1+\delta_jX, g_2+\delta_jY) =  e_1^{\mathcal{O}_{j+1}}(e_1^{\mathcal{O}_{j+1}}(g_1 + \delta_jX, g_2+\delta_jY), g_2 + \delta_jY)\\
&=& e_1^{\mathcal{O}_{j+1}}(e_1^{\mathcal{O}_{j+1}}(g_1, g_2) + \delta_j De_1^{\mathcal{O}_{j+1}}(X, Y), g_2 + \delta_jY)\\
& =& (e_1^{\mathcal{O}_{j+1}}(g_1, g_2) + \delta_j De_1^{\mathcal{O}_{j+1}}(X, Y))(g_2 + \delta_jY)-(g_2 + \delta_jY)(e_1^{\mathcal{O}_{j+1}}(g_1, g_2) + \delta_j De_1^{\mathcal{O}_{j+1}}(X, Y))\\
&=&  {\small (e_1^{\mathcal{O}_{j+1}}(g_1, g_2)g_2-g_2 e_1^{\mathcal{O}_{j+1}}(g_1, g_2)) + \delta_j(e_1^{\mathcal{O}_{j+1}}(g_1, g_2)Y + De_1^{\mathcal{O}_{j+1}}(X, Y)g_2 - g_2 De_1^{\mathcal{O}_{j+1}}(X, Y)  
-Y e_1^{\mathcal{O}_{j+1}}(g_1, g_2))}\\
& =& e_2^{\mathcal{O}_{j+1}}(g_1, g_2) + \delta_j\left(e_1^{\mathcal{O}_{j+1}}(De_1^{\mathcal{O}_{j+1}}(X, Y), g_2)-e_1^{\mathcal{O}_{j+1}}(Y, e_1^{\mathcal{O}_{j+1}}(g_1, g_2)) \right)\\
&=& e_2^{\mathcal{O}_{j+1}}(g_1,g_2)+\delta_j De_2^{\mathcal{O}_{j+1}}(X,Y)
\end{eqnarray*}
where $De_2^{\mathcal{O}_{j+1}}\colon \mathfrak{sl}_2(\mathcal{O}_{j+1})\times \mathfrak{sl}_2(\mathcal{O}_{j+1})\rightarrow\mathfrak{sl}_2(\mathcal{O}_{j+1})$ is defined by \begin{equation}{\label{der 2}}
(X, Y)\mapsto \left( e_1^{\mathcal{O}_{j+1}}(De_1^{\mathcal{O}_{j+1}}(X,Y),g_2)-e_1^{\mathcal{O}_{j+1}}(Y,e_1^{\mathcal{O}_{j+1}}(g_1,g_2))\right).
\end{equation}
Now, we claim that,  
$$De_{m+1}^{\mathcal{O}_{j+1}}\colon \mathfrak{sl}_2(\mathcal{O}_{j+1}) \times \mathfrak{sl}_2(\mathcal{O}_{j+1}) \rightarrow \mathfrak{sl}_2(\mathcal{O}_{j+1})$$ is given by \begin{equation}{\label{der 3}}
(X, Y)\mapsto \left( e_1^{\mathcal{O}_{j+1}}(De_m^{\mathcal{O}_{j+1}}(X,Y),g_2)-e_1^{\mathcal{O}_{j+1}}(Y,e_m^{\mathcal{O}_{j+1}}(g_1,g_2))\right)\ \text{for any}\ m\geq 1.
\end{equation}
This can easily be seen by the following inductive step: 

Let $$ e_m^{\mathcal{O}_{j+1}}(g_1+\delta_jX, g_2+\delta_jY)= e_m^{\mathcal{O}_{j+1}}(g_1, g_2)+\delta_j De_m^{\mathcal{O}_{j+1}}(X,Y).$$
Then 

\begin{eqnarray*}
&& e_{m+1}^{\mathcal{O}_{j+1}}(g_1+\delta_jX, g_2+\delta_jY) =  e_1^{\mathcal{O}_{j+1}}(e_m^{\mathcal{O}_{j+1}}(g_1 + \delta_jX, g_2+\delta_jY), g_2 + \delta_jY)\\
&=& e_1^{\mathcal{O}_{j+1}}(e_m^{\mathcal{O}_{j+1}}(g_1, g_2) + \delta_j De_m^{\mathcal{O}_{j+1}}(X, Y), g_2 + \delta_jY)\\
& =& (e_m^{\mathcal{O}_{j+1}}(g_1, g_2) + \delta_j De_m^{\mathcal{O}_{j+1}}(X, Y))(g_2 + \delta_jY)-(g_2 + \delta_jY)(e_m^{\mathcal{O}_{j+1}}(g_1, g_2) + \delta_j De_m^{\mathcal{O}_{j+1}}(X, Y))\\
&=&  {\small (e_m^{\mathcal{O}_{j+1}}(g_1, g_2)g_2-g_2 e_m^{\mathcal{O}_{j+1}}(g_1, g_2)) + \delta_j(e_m^{\mathcal{O}_{j+1}}(g_1, g_2)Y + De_m^{\mathcal{O}_{j+1}}(X, Y)g_2 - g_2 De_m^{\mathcal{O}_{j+1}}(X, Y)  
-Y e_m^{\mathcal{O}_{j+1}}(g_1, g_2))}\\
& =& e_{m+1}^{\mathcal{O}_{j+1}}(g_1, g_2) + \delta_j\left(e_1^{\mathcal{O}_{j+1}}(De_m^{\mathcal{O}_{j+1}}(X, Y), g_2)-e_1^{\mathcal{O}_{j+1}}(Y, e_m^{\mathcal{O}_{j+1}}(g_1, g_2)) \right).
\end{eqnarray*}

This $De_{m+1}^{\mathcal{O}_{j+1}}$ is called the \emph{$(m+1)$-th derivative map} corresponding to $e_{m+1}^{\mathcal{O}_{j+1}}$ at $(g_1, g_2)$.

\begin{remark}{\label{scalarmatrixderivativemap}}
If $X$ and $Y$ are scalar matrices in $\M_2(\mathcal{O}_{j+1})$ (not necessary that $X, Y\in \mathfrak{sl}_2(\mathcal{O}_{j+1})$) then $De_{m+1}^{\mathcal{O}_{j+1}}(X, Y)$ (here $De_{m+1}^{\mathcal{O}_{j+1}}\colon \M_2(\mathcal{O}_{j+1}) \times \M_2(\mathcal{O}_{j+1}) \rightarrow \M_2(\mathcal{O}_{j+1})$ defined recursively as in the case $\mathfrak{sl}_2(\mathcal{O}_{j+1})$) at $(g_1, g_2)$ is $\mathbf{0}$. This easily follows from the fact that  $De_1^{\mathcal{O}_{j+1}}(X,Y)$ at $(g_1, g_2)$ is $\mathbf{0}$ in this case.
\end{remark}
\begin{lemma}{\label{lem-adelic 2}}
Let $\mathcal{O}_2$ be a local principal ideal ring of length $2$ with residue field $k$ of characteristic $\neq 2$. Let $\mathbf{\bar g}_1, \mathbf{\bar g}_2 \in \SL_2(k)$ be such that the groups $Z (\bar e_1(\mathbf{\bar g}_1, \mathbf{\bar g}_2)) \cap Z(\mathbf{\bar g}_2) = \{\mathbf{0}\}$ in $\mathfrak{sl}_2(k)$. Further, assume that $Z\left(\mathbf{\bar g}_2\right) \cap Z\left(\mathbf{\bar g}_2\right)^{\perp}$ is trivial in $\mathfrak{sl}_2(k)$. Then, for any lift $\mathbf{h}$ of $\bar e_{2}(\mathbf{\bar g}_1, \mathbf{\bar g}_2)$ in $\mathfrak{sl}_2^{\circ}(\mathcal O_2)$ there exists $\mathbf{h}_1, \mathbf{ h}_2 \in \M_2(\mathcal{O}_2)$ such that $\mathbf{h} = e_{2}^{\mathcal{O}_2}(\mathbf{h}_1, \mathbf{h}_2)$ where $\mathbf{h}_1, \mathbf{ h}_2$ are lifts of $\mathbf{\bar g}_1$ and $\mathbf{\bar g}_2$ respectively.
\end{lemma}
\begin{proof}
It is enough to show the surjectivity of the derivative map at $(\mathbf{\bar g}_1, \mathbf{\bar g}_2)$, corresponding to $\bar e_{2}$ at the Lie algebra level. Take any lift $\mathbf{g}$ of $\bar e_{2}(\mathbf{\bar g}_1, \mathbf{\bar g}_2)$ in $\mathfrak{sl}_2^{\circ}(\mathcal{O}_2)$. As explained above, there exists $\mathbf{g}_1, \mathbf{g}_2$ (lifts of $\mathbf{\bar g}_1, \mathbf{\bar g}_2$ respectively) in $\M_2(\mathcal{O}_2)$ such that $\mathbf{g} = e_{2}^{\R}(\mathbf{g}_1, \mathbf{g}_2)+ \pi_2 \mathcal{D}$, for some $\mathcal D \in \M_2(\mathcal{O}_2)$, where $\tr(\mathcal{D})\in \mathfrak{m}_2$ (i.e., $\mathcal{D} \in \mathfrak{sl}_2(\mathcal{O}_2)$) and $\pi_2$ is the generator of $\mathfrak{m}_2$ with $\pi_2^2 = 0$. Let $X, Y\in \mathfrak{sl}_2(\mathcal{O}_2)$. 
    
As explained above, we have $e_2^{\mathcal{O}_{j+1}}(\mathbf g_1 + \delta_jX, \mathbf g_2 + \delta_jY) = e_2^{\mathcal{O}_{j+1}}(\mathbf g_1, \mathbf g_2) + \pi_2 De_2^{\mathcal{O}_{j+1}}(X, Y)$ where the derivative of the map $e_2$, denoted as $De^{\R}_2$, at $(\mathbf{g}_1, \mathbf{g}_2)$ is obtained from the Equation~(\ref{der 2}) above:
$$De^{\R}_2\colon  (X,Y)\mapsto \left( e_1^{\mathcal{O}_{2}}(De_1^{\mathcal{O}_{2}}(X,Y),\mathbf g_2)-e_1^{\mathcal{O}_{2}}(Y,e_1^{\mathcal{O}_{2}}(\mathbf g_1,\mathbf g_2))\right).$$ 
The notation $De^{\R}_2$ indicates that we are working over $\mathcal{O}_2$. We need to show that this map is surjective. 

The non-degenerate symmetric bilinear form $\langle P, Q\rangle = \tr(PQ)$ is conjugation invariant, i.e., for any $\bar g\in \GL_2(k)$, $\langle \bar gP\bar g^{-1}, \bar gQ\bar g^{-1}\rangle = \tr(\bar gPQ\bar g^{-1})=\tr(PQ)=\left\langle P, Q\right\rangle$. We use the notation $\overline{De_2}$ to denote the same derivative map after reduction mod $\mathfrak{m}_2$, i.e., over $\mathfrak{sl}_2(k)$. We show, under the given assumptions, that the orthogonal complement $\im(\overline{De}_2)^{\perp}$ is trivial. For this, let $T$ be any vector in $\im(\overline{De}_2)^{\perp}$. Then, $\langle T, \overline{De}_2(X, Y)\rangle = 0 \ \forall X, Y\in \mathfrak{sl}_2(k)$. Now,
\begin{eqnarray*}
\overline{D}e_2(X,Y)&=&  \bar e_1\left(\overline{De}_1(X,Y),\mathbf {\bar g}_2\right)-\bar e_1\left(Y,\bar e_1(\mathbf {\bar g}_1,\mathbf {\bar g}_2)\right).
\end{eqnarray*}

We consider the following two cases.

\textbf{Case I. When $(Y = \mathbf{0})$:} In this case $\langle \overline{De}_2(X, Y), T \rangle = 0$. That is, for all $X$, we have
\begin{eqnarray*}
0 &=& \langle \overline{De}_2(X,\mathbf{0}),T \rangle = \left\langle \bar e_1\left(\overline{De}_1(X,\mathbf{0}),\mathbf {\bar g}_2\right),T\right\rangle-\left\langle \bar e_1\left(\mathbf{0},\bar e_1(\mathbf {\bar g}_1,\mathbf {\bar g}_2)\right),T\right\rangle  \\
&=& \left\langle\bar e_1\left(\overline{De}_1(X,\mathbf{0}),\mathbf {\bar g}_2\right),T\right\rangle\\
&=& \left\langle \bar e_1(\bar e_1(X,\mathbf{\bar g}_2),\mathbf{\bar g}_2),T\right\rangle\ \ \ \ \ [\text{As} \ \overline{De}_1(X,\mathbf{0})=\bar e_1(X,\mathbf{\bar g}_2),\ \text{by}\ \text{Equation}~(\ref{der 1})]\\
&=& \left\langle \bar e_1(X,\mathbf{\bar g}_2)\mathbf{\bar g}_2-\mathbf{\bar g}_2\bar e_1(X,\mathbf{\bar g}_2),T\right\rangle 
= \left\langle X\mathbf{\bar g}_2^2-\mathbf{\bar g}_2X\mathbf{\bar g}_2- \mathbf{\bar g}_2X\mathbf{\bar g}_2+\mathbf{\bar g}_2^2X,T\right\rangle\\
&=& \tr(X\mathbf{\bar g}_2^2T-\mathbf{\bar g}_2X\mathbf{\bar g}_2T- \mathbf{\bar g}_2X\mathbf{\bar g}_2T+\mathbf{\bar g}_2^2XT) 
= \tr\left(X(\mathbf{\bar g}_2^2T-\mathbf{\bar g}_2T\mathbf{\bar g}_2-\mathbf{\bar g}_2T\mathbf{\bar g}_2+T\mathbf{\bar g}_2^2)\right) \\
&=& \left\langle X, \mathbf{\bar g}_2U_1-U_1\mathbf{\bar g}_2\right\rangle \ \ \ \ \ \text{Where}\ U_1=[\mathbf{\bar g}_2,T] \text{\ and use\ } tr(AB)=tr(BA) .
\end{eqnarray*} 
Using the non-degeneracy of the bilinear form we get $\mathbf{\bar g}_2U_1 - U_1\mathbf{\bar g}_2 = \mathbf{0}$. Hence $U_1\in Z(\mathbf{g}_2)$. As $U_1=[\mathbf{\bar g}_2, T]$ therefore $\left\langle U_1, \gamma\right \rangle = \tr(\mathbf{\bar g}_2T \gamma -T \mathbf{\bar g}_2 \gamma) = 0$ for all $\gamma\in Z(\mathbf{\bar g}_2)$. This further implies $U_1\in Z(\mathbf{\bar g}_2)^{\perp}$. Therefore, we obtain:
\begin{equation}{\label{ortho1}}
U_1\in Z(\mathbf{\bar g}_2)\cap Z(\mathbf{\bar g}_2)^{\perp}.
\end{equation}
 
\textbf{Case II. When $(X=\mathbf{0})$:} In  this case $\langle T, \overline{De}_2(X, Y)\rangle = 0$. That is for all $Y$, we have,
\begin{eqnarray*}
0&=& \langle \overline{De}_2(\mathbf{0},Y),T\rangle = \langle  \bar e_1\left(\overline{De}_1(\mathbf{0},Y),\mathbf {\bar g}_2\right)-\bar e_1\left(Y,\bar e_1(\mathbf {\bar g}_1,\mathbf {\bar g}_2)\right),T \rangle \\
&=& \left\langle  \bar e_1\left(\overline{De}_1(\mathbf{0},Y),\mathbf {\bar g}_2\right),T\right\rangle-\left\langle \bar e_1\left(Y,\bar e_1(\mathbf {\bar g}_1,\mathbf {\bar g}_2)\right),T\right\rangle \\
&=& \left\langle \overline{De}_1(\mathbf{0},Y)\mathbf{\bar g}_2-\mathbf{\bar g}_2\overline{De}_1(\mathbf{0},Y),T\right\rangle-\left\langle Y\bar e_1(\mathbf {\bar g}_1,\mathbf {\bar g}_2)-\bar e_1(\mathbf {\bar g}_1,\mathbf {\bar g}_2)Y,T \right\rangle\\
&=& \left\langle \overline{De}_1(\mathbf{0},Y),\mathbf{\bar g}_2T-T\mathbf{\bar g}_2\right\rangle-\left\langle Y, \bar e_1(\mathbf {\bar g}_1,\mathbf {\bar g}_2)T-T\bar e_1(\mathbf {\bar g}_1,\mathbf {\bar g}_2)\right\rangle
= \left\langle\overline{De}_1(\mathbf{0},Y),U_1 \right\rangle-\left\langle Y, [\bar e_1(\mathbf {\bar g}_1,\mathbf {\bar g}_2),T]\right\rangle.
\end{eqnarray*}
This further implies 
\begin{equation}{\label{ortho2}}
\left\langle Y, [\bar e_1(\mathbf {\bar g}_1,\mathbf {\bar g}_2), T] \right \rangle = \left \langle \overline{De}_1(\mathbf{0}, Y), U_1 \right \rangle.
\end{equation}
Now, we claim that $U_1= \mathbf{0}$. As $Z(\mathbf{\bar g}_2) \cap Z(\mathbf{\bar g}_2)^{\perp} = \{\mathbf{0}\}$ by assumption, therefore Equation~(\ref{ortho1}) implies $U_1 = \mathbf{0}$. Hence, the claim is true.

Since $U_1=[\mathbf{\bar g}_2,T]=0$, therefore $T\in Z(\mathbf{\bar g}_2)$. Moreover, substituting $U_1=\mathbf{0}$ in Equation~(\ref{ortho2}) we obtain $ [\bar e_1(\mathbf {\bar g}_1, \mathbf {\bar g}_2), T] = \mathbf{0}$, by the non-degeneracy of the bilinear form defined above. Therefore, we have the following properties of $T$: (1) $T\in Z(\mathbf{\bar g}_2)$ and (2)  $ [\bar e_1(\mathbf {\bar g}_1, \mathbf {\bar g}_2),T] = \mathbf{0}$ implies $T\in Z(\bar e_1(\mathbf {\bar g}_1, \mathbf {\bar g}_2))$. 
These together imply $T\in  Z(\mathbf{\bar g}_2)\cap Z(\bar e_1(\mathbf {\bar g}_1, \mathbf {\bar g}_2))$. By the assumption $ Z (\bar e_1(\mathbf{\bar g}_1, \mathbf{\bar g}_2)) \cap Z(\mathbf{\bar g}_2)=\{\mathbf{0}\}$. This gives us $T=\mathbf{0}$.

Therefore, under the given assumptions, we obtain $\im(\overline{De}_2)^\perp \cap \mathfrak{sl}_2(k) = \{\mathbf{0}\}$. The dimension formula $dim(\im(\overline{De}_2)) + dim(\im(\overline{De}_2)^\perp) = dim(\mathfrak{sl}_2(k))$ implies $coker(\overline{De}_2) = \mathfrak{sl}_2(k)/ \im(\overline{De}_2) = \{0\}$. 
Note that $\im(De^{\mathcal O_2}_2)$ is an $\mathcal{O}_2$ submodule of the $\mathcal{O}_2$ module $\mathfrak{sl}_2(\mathcal{O}_2)$. Now, consider the short exact sequence 
\[\begin{tikzcd}
\mathbf{0} & {\im(De^{\R}_2)} & {\mathfrak{sl}_2(\mathcal{O}_2)} & {\mathfrak{sl}_2(\mathcal{O}_2)/\im(De^{\R}_2)} & \mathbf{0}.
	\arrow[from=1-1, to=1-2]
	\arrow[from=1-2, to=1-3]
	\arrow[from=1-3, to=1-4]
	\arrow[from=1-4, to=1-5]
\end{tikzcd}\]
Tensoring with $\mathcal{O}_2/\mathfrak{m}_2$, and using the right exactness we obtain $\left(\mathfrak{sl}_2(\mathcal{O}_2)/\im(De^{\R}_2)\right) \bigotimes_{\mathcal{O}_2} \mathcal{O}_2/ \mathfrak{m}_2\cong coker(\overline{De}_2) = \{\mathbf{0}\}$. Let us denote $\mathcal{M} = \mathfrak{sl}_2(\mathcal{O}_2)/ \im(De^{\R}_2)$; clearly it is a finitely generated $\mathcal{O}_2$ module. Therefore $\mathcal{M}/ \mathfrak{m}_2\mathcal{M} = \{\mathbf{0}\}$. Invoking Proposition 2.6 \cite{Atiyahmacdonald} we obtain $\im(De^{\R}_2) = \mathfrak{sl}_2(\mathcal{O}_2)$. This completes the proof.
\end{proof}

\begin{proposition}{\label{adelic 2}}
Let $\mathcal{O}$ be a local principal ideal ring, complete with respect to its maximal ideal $\mathfrak{m}=(\pi)$ with residue field $k$ of characteristic $\neq 2$. Let $\mathbf{\bar g}_1, \mathbf{\bar g}_2 \in \M_2(k)$ be such that  the groups $ Z (\bar e_m(\mathbf{\bar g}_1,\mathbf{\bar g}_2))\cap Z(\mathbf{\bar g}_2)=\{\mathbf{0}\}$ in $\mathfrak{sl}_2(k)$. Further, assume that $Z\left(\mathbf{\bar g}_2\right)\cap Z\left(\mathbf{\bar g}_2\right)^{\perp}$ is trivial in $\mathfrak{sl}_2(k)$. Then, for any lift $\mathbf{h}$ of $\bar e_{m+1}(\mathbf{\bar g}_1, \mathbf{\bar g}_2)$ in $\mathfrak{sl}_2^{\circ}(\mathcal O)$ there exists $\mathbf{h}_1, \mathbf{ h}_2\in\M_2(\mathcal{O})$ such that $\mathbf{h} = e_{m+1}^{\mathcal{O}}(\mathbf{h}_1, \mathbf{h}_2)$ where $\mathbf{h}_1, \mathbf{ h}_2$ are lifts of $\mathbf{\bar g}_1$ and $\mathbf{\bar g}_2$ respectively.
\end{proposition}
\begin{proof}
It is enough to show the surjectivity of the derivative map of $\bar e_{m+1}$ at $(\mathbf{\bar g}_1, \mathbf{\bar g}_2)$ at the Lie algebra level. At first, we work over $\mathcal O_2$. Take any lift $\mathbf{g}$ of $\bar e_{m+1}(\mathbf{\bar g}_1, \mathbf{\bar g}_2)$ in $\M_2(\mathcal{O}_2)$. Then (as explained before Lemma~\ref{lem-adelic 2}) there exists $\mathbf{g}_1, \mathbf{g}_2$, lifts of $\mathbf{\bar g}_1, \mathbf{\bar g}_2$ respectively, in $\M_2(\mathcal{O}_2)$ such that $\mathbf{g}= e_{m+1}^{\R}(\mathbf{g}_1, \mathbf{g}_2)+ \pi_2 \mathcal{D}$ for some $\mathcal D$ where $\tr(\mathcal{D}) \in \mathfrak{m}_2$. Recall that $\pi_2$ is the generator of $\mathfrak{m}_2$ with $\pi_2^2 = 0$. Let $X, Y \in \mathfrak{sl}_2(\mathcal{O}_2)$. When $m=1$, we have proved this in Lemma~\ref{lem-adelic 2}.

{\bf Step I:}  Along the line of proof in the previous Lemma, we look at the equation $\langle \overline{De}_{m+1}(X, Y), T \rangle = 0$ for $(m+1)$-th Engel word $ e_{m+1}^{\R}$ for a given $m\geq 2$, where $T\in \im(\overline{De}_{m+1})^{\perp}$; when we put $Y=\mathbf{0}$ and $X=\mathbf{0}$ separately. Let us list down the two equations for $De^{\R}_{m+1} \colon  \mathfrak{sl}_2(\mathcal{O}_2) \times \mathfrak{sl}_2(\mathcal{O}_2) \rightarrow \mathfrak{sl}_2(\mathcal{O}_2)$ defined by (see Equation (\ref{der 3})) 
$$  (X,Y)\mapsto \left( e_1^{\mathcal{O}_{2}}(De_m^{\mathcal{O}_{2}}(X,Y),\mathbf{g}_2)-e_1^{\mathcal{O}_{2}}(Y,e_m^{\mathcal{O}_{2}}(\mathbf{g}_1,\mathbf{g}_2))\right)$$  
as follows.

\textbf{Case I. When $ (Y=\mathbf{0})$:} In this case $\langle \overline{De}_{m+1}(X,Y), T \rangle = 0$.
That is, for all $X$, we have
\begin{eqnarray}{\label{eqnortho}}\nonumber 
0=\langle\overline{De}_{m+1}(X,\mathbf{0}),T \rangle &=& \left\langle\left( \bar e_1(\overline{De}_m(X,\mathbf{0}),\mathbf{\bar g}_2)-\bar e_1(\mathbf{0},\bar e_m(\mathbf{\bar g}_1,\mathbf{\bar g}_2))\right),T\right\rangle = \left\langle  \bar e_1(\overline{De}_m(X,\mathbf{0}),\mathbf{\bar g}_2),T\right\rangle \\
&=& \left\langle \overline{De}_m(X,\mathbf{0})\mathbf{\bar g}_2-\mathbf{\bar g}_2\overline{De}_m(X,\mathbf{0}),T\right\rangle = \left\langle \overline{De}_m(X,\mathbf{0}),U_1\right\rangle\ \ \ [\text{Where}\ U_1=[\mathbf{\bar g}_2,T]].
\end{eqnarray} 
Now,
\begin{eqnarray*}
\left\langle \overline{De}_m(X,\mathbf{0}),U_1\right\rangle&=& \left\langle\left( \bar e_1(\overline{De}_{m-1}(X,\mathbf{0}),\mathbf{\bar g}_2)-\bar e_1(\mathbf{0},\bar e_{m-1}(\mathbf{\bar g}_1,\mathbf{\bar g}_2))\right),T\right\rangle \\\nonumber
&=& \left\langle  \bar e_1(\overline{De}_{m-1}(X,\mathbf{0}),\mathbf{\bar g}_2),T\right\rangle = \left\langle \overline{De}_{m-1}(X,\mathbf{0})\mathbf{\bar g}_2-\mathbf{\bar g}_2\overline{De}_{m-1}(X,\mathbf{0}),T\right\rangle\\
&=& \left\langle \overline{De}_{m-1}(X,\mathbf{0}),U_2\right\rangle\ \ \ [\ \text{Where}\ U_2=[\mathbf{\bar g}_2,U_1]=[\mathbf{\bar g}_2,[\mathbf{\bar g}_2,T]]\ ]\\
& & \ \ \vdots\  \ \ \ \ \ \ \ \ \vdots\ \ \ \ \ \ \ \ \ \vdots\\
&=&  \left\langle \overline{De}_{1}(X,\mathbf{0}),U_m\right\rangle\ \ \ [\ \text{Where}\ U_m=[\mathbf{\bar g}_2,U_{m-1}]=[\underbrace{\mathbf{\bar g}_2[\cdots[\mathbf{\bar g}_2,[\mathbf{\bar g}_2}_{m\ \text{times}},T]]\cdots]\ ]\\
&=& \left\langle X,U_{m+1}\right\rangle \ \ \ [\ \text{Where}\ U_{m+1}=[\mathbf{\bar g}_2,U_{m}]\ \text{and}\ \overline{De}_{1}(X,\mathbf{0})=X\mathbf{\bar g}_2-\mathbf{\bar g}_2X\ ]
\end{eqnarray*}

Therefore the Equation~(\ref{eqnortho}) finally turns out to be: $$\left\langle X,U_{m+1}\right\rangle=0\ \ \forall X\in \mathfrak{sl}_2(k);\ \text{where}\ U_{m+1}=[\mathbf{\bar g}_2,U_{m}]=[\underbrace{\mathbf{\bar g}_2[\cdots[\mathbf{\bar g}_2,[\mathbf{\bar g}_2}_{m+1\ \text{times}},T]]\cdots].$$
The non-degeneracy of the bilinear form implies $U_{m+1}=\mathbf{0}$. Thus, we have the following properties of $U_m$: (1)  $U_{m+1}=\mathbf{0}$ implies $[\mathbf{\bar g}_2,U_{m}] = \mathbf{0}$ and hence $U_m\in Z(\mathbf{\bar g}_2)$. (2) As $U_m=[\mathbf{\bar g}_2,U_{m-1}]$ therefore $U_m\in Z(\mathbf{\bar g}_2)^{\perp}$.

These properties yield 
\begin{equation}{\label{eqnU_m}}
U_m\in Z(\mathbf{\bar g}_2)\cap Z(\mathbf{\bar g}_2)^{\perp}.
\end{equation}
However, by assumption $Z(\mathbf{\bar g}_2)\cap Z(\mathbf{\bar g}_2)^{\perp}=\{\mathbf{0\}}$, therefore, from Equation~(\ref{eqnU_m}) we obtain $U_m=\mathbf{0}$. Continuing this process, inductively we obtain $U_1=\mathbf{0}$, and hence 
\begin{equation}{\label{T in Z g2}}
    T\in Z(\mathbf{\bar g}_2).
\end{equation} 

\textbf{Case II. When $ (X=\mathbf{0})$ :} 
In this case $\langle T, \overline{De}_{m+1}(X,Y)\rangle = 0$. That is, for all $Y\in \mathfrak{sl}_2(k)$ we have 
\begin{eqnarray*}
0 &=& \left\langle \overline{De}_{m+1}(\mathbf{0}, Y),T \right\rangle =\langle  \bar e_1\left(\overline{De}_m(\mathbf{0},Y),\mathbf {\bar g}_2\right)-\bar e_1\left(Y,\bar e_m(\mathbf {\bar g}_1,\mathbf {\bar g}_2)\right),T \rangle \\
&=& \left\langle  \bar e_1\left(\overline{De}_m(\mathbf{0},Y),\mathbf {\bar g}_2\right),T\right\rangle-\left\langle \bar e_1\left(Y,\bar e_m(\mathbf {\bar g}_1,\mathbf {\bar g}_2)\right),T\right\rangle \\
&=& \left\langle \overline{De}_m(\mathbf{0},Y)\mathbf{\bar g}_2-\mathbf{\bar g}_2\overline{De}_m(\mathbf{0},Y),T\right\rangle-\left\langle Y\bar e_m(\mathbf {\bar g}_1,\mathbf {\bar g}_2)-\bar e_m(\mathbf {\bar g}_1,\mathbf {\bar g}_2)Y,T \right\rangle\\
&=& \left\langle \overline{De}_m(\mathbf{0},Y),\mathbf{\bar g}_2T-T\mathbf{\bar g}_2\right\rangle-\left\langle Y, \bar e_m(\mathbf {\bar g}_1,\mathbf {\bar g}_2)T-T\bar e_m(\mathbf {\bar g}_1,\mathbf {\bar g}_2)\right\rangle
= \left\langle\overline{De}_m(\mathbf{0},Y),U_1 \right\rangle-\left\langle Y, [\bar e_m(\mathbf {\bar g}_1,\mathbf {\bar g}_2),T]\right\rangle.
\end{eqnarray*}
This further implies 
\begin{equation}{\label{ortho3}}
\left\langle Y, [\bar e_m(\mathbf {\bar g}_1,\mathbf {\bar g}_2), T] \right\rangle = \left\langle \overline{De}_m(\mathbf{0}, Y), U_1 \right \rangle \ \ \forall \ Y\in \mathfrak{sl}_2(k).
\end{equation}

{\bf Step II:} Using the Equation~(\ref{eqnU_m}) we already have proved $U_1=\mathbf{0}$. Substituting this value in the Equation~(\ref{ortho3}), we obtain $\left\langle Y, [\bar e_m(\mathbf {\bar g}_1,\mathbf {\bar g}_2),T]\right\rangle=0 \ \ \forall \ Y\in \mathfrak{sl}_2(k).$ The non-degeneracy of the bilinear form implies 
\begin{equation}
{\label{complement}}
    T\in Z(\bar e_m(\mathbf {\bar g}_1,\mathbf {\bar g}_2)).
\end{equation}
Therefore Equation~(\ref{T in Z g2}) and 
Equation~(\ref{complement}) together gives $T\in Z(\mathbf{\bar g}_2)\cap Z(\bar e_m(\mathbf {\bar g}_1,\mathbf {\bar g}_2))$.

Now, by the assumption, we have $Z(\mathbf{\bar g}_2) \cap Z(\bar e_m(\mathbf {\bar g}_1,\mathbf {\bar g}_2))=\{\mathbf{0}\}$. This immediately implies $T = \mathbf{0}$ and hence $\im(\overline{De}_{m+1})^\perp \cap \mathfrak{sl}_2(k) = \{\mathbf{0}\}$. Therefore, the formula $dim(\im(\overline{De}_{m+1})) + dim(\im(\overline{De}_{m+1})^\perp) = dim (\mathfrak{sl}_2(k))$ implies $coker(\overline{De}_{m+1}) = \mathfrak{sl}_2(k)/ \im(\overline{De}_{m+1}) = \{\mathbf{0}\}$. Now consider the short exact sequence 
\[\begin{tikzcd}
	\mathbf{0} & {\im(De^{\R}_{m+1})} & {\mathfrak{sl}_2(\mathcal{O}_2)} & {\mathfrak{sl}_2(\mathcal{O}_2)/\im(De^{\R}_{m+1})} & \mathbf{0}.
	\arrow[from=1-1, to=1-2]
	\arrow[from=1-2, to=1-3]
	\arrow[from=1-3, to=1-4]
	\arrow[from=1-4, to=1-5]
\end{tikzcd}\]
Tensoring with $\mathcal{O}_2/\mathfrak{m}_2$, and using right exactness we obtain $$\mathfrak{sl}_2(\mathcal{O}_2)/\im(De^{\R}_{m+1})\bigotimes_{\mathcal{O}_2}\mathcal{O}_2/\mathfrak{m}_2\cong coker(\overline{De}_{m+1}) = \{\mathbf{0}\}.$$ 
Let us denote $\mathcal{M} = \mathfrak{sl}_2(\mathcal{O}_2)/ \im(De^{\R}_{m+1})$ which is an $\mathcal{O}_2$ module. Therefore $\mathcal{M}/ \mathfrak{m}_2\mathcal{M} = \{\mathbf{0}\}$. Invoking \cite[Proposition 2.6]{Atiyahmacdonald} we obtain $\im(De^{\R}_{m+1}) = \mathfrak{sl}_2(\mathcal{O}_2)$.

{\bf Step III:} Therefore, if $A\in \M_2(\mathcal{O})$ such that $\tr(A)=0$, is a lift of $\bar e_{m+1}(\mathbf{\bar g}_1, \mathbf{\bar g}_2)$ then there exists $(\mathbf{h}_1, \mathbf{h}_2)\in \M_2(\mathcal{O}_2)^2$ such that $\mathbf{h}_i$ are lifts of corresponding $\mathbf{\bar g}_i$ for $i=1,2$ and  $e_{m+1}(\mathbf{h}_1, \mathbf{h}_2) = A_2$. The kernel of the reduction map $\M_2(\mathcal{O}_{j+1}) \rightarrow \M_2(\mathcal{O}_j)$ is $\M_2(\mathfrak{m}_{j+1}^j)$ where $\mathfrak{m}_{j+1}^j=(\delta_j)$; $\delta_j$ is the generator of $ker(\theta_j)$. Note that $(ker(\theta_j))^2=0$ implies $\delta_j^2=0$. Therefore in the context of lifting of elements from $\im(e_{m+1}^{\mathcal{O}_j})$ to $\im(e_{m+1}^{\mathcal{O}_{j+1}})$, one can observe $e_{m+1}^{\mathcal{O}_{j+1}}(B_1 + \delta_jX, B_2 + \delta_jY) = e_{m+1}^{\mathcal{O}_{j+1}}(B_1, B_2) + \delta_j De_{m+1}^{\mathcal{O}_{j+1}}(X, Y)$;  because:
\begin{eqnarray*}
&& e_{m+1}^{\mathcal{O}_{j+1}}(B_1 + \delta_jX, B_2 + \delta_jY) =  e_1^{\mathcal{O}_{j+1}}(e_m^{\mathcal{O}_{j+1}}(B_1 + \delta_jX, B_2+\delta_jY), B_2 + \delta_jY)\\
&=& e_1^{\mathcal{O}_{j+1}}(e_m^{\mathcal{O}_{j+1}}(B_1, B_2) + \delta_j De_m^{\mathcal{O}_{j+1}}(X, Y), B_2 + \delta_jY)\\
& =& (e_m^{\mathcal{O}_{j+1}}(B_1, B_2) + \delta_j De_m^{\mathcal{O}_{j+1}}(X, Y))(B_2 + \delta_jY)-(B_2 + \delta_jY)(e_m^{\mathcal{O}_{j+1}}(B_1, B_2) + \delta_j De_m^{\mathcal{O}_{j+1}}(X, Y))\\
&=&   (e_m^{\mathcal{O}_{j+1}}(g_1, B_2)B_2-B_2 e_m^{\mathcal{O}_{j+1}}(B_1, B_2)) + \delta_j(e_m^{\mathcal{O}_{j+1}}(B_1, B_2)Y + De_m^{\mathcal{O}_{j+1}}(X, Y)B_2 - B_2 De_m^{\mathcal{O}_{j+1}}(X, Y)  \\ && 
-Y e_m^{\mathcal{O}_{j+1}}(B_1, B_2))\\
& =& e_{m+1}^{\mathcal{O}_{j+1}}(B_1, B_2) + \delta_j\left(e_1^{\mathcal{O}_{j+1}}(De_m^{\mathcal{O}_{j+1}}(X, Y), B_2)-e_1^{\mathcal{O}_{j+1}}(Y, e_m^{\mathcal{O}_{j+1}}(B_1, B_2)) \right)\\
&=& e_{m+1}^{\mathcal{O}_{j+1}}(B_1, B_2) + \delta_j De_{m+1}^{\mathcal{O}_{j+1}}(X, Y)
\end{eqnarray*}
where $B_1, B_2\in \M_2(\mathcal{O}_{j+1})$ are the lifts of $\mathbf{\bar g}_1$ and $\mathbf{\bar g}_2$ respectively (achieved through successive lifting of $\mathbf{\bar g}_1$ and $\mathbf{\bar g}_2$ at each $\mathcal{O}_{\ell}$ level for $\ell \leq j+1$), moreover $X, Y\in \mathfrak{sl}_2(\mathcal{O}_{j+1})$ and $De_{m+1}^{\mathcal{O}_{j+1}}\colon  \mathfrak{sl}_2(\mathcal{O}_{j+1}) \times \mathfrak{sl}_2(\mathcal{O}_{j+1}) \rightarrow \mathfrak{sl}_2( \mathcal{O}_{j+1}) $ is the derivative map corresponding to $e_{m+1}$ at $(B_1, B_2)$ in $j+1$-th level defined earlier.

As $\mathfrak{sl}_2(\mathcal{O}_{j+1})/ \im(De^{\mathcal{O}_{j+1}}_{m+1})\bigotimes_{\mathcal{O}_{j+1}}\mathcal{O}_{j+1}/\mathfrak{m}_{j+1}\cong coker(\overline{De}_{m+1}) = \{\mathbf{0}\}$, therefore Nakayama's lemma \cite{Atiyahmacdonald} guarantees that all these $De_{m+1}^{\mathcal{O}_{j+1}}$ are surjective for $j\geq 1$. Therefore there exists $(P_{j+1}, Q_{j+1})\in \M_2(\mathcal{O}_{j+1})^2$ which is the lift of corresponding $(P_j, Q_j)\in \M_2(\mathcal{O}_j)^2$ such that $A_{j+1} = e_{m+1}^{\mathcal{O}_{j+1}}(P_{j+1}, Q_{j+1})$ for each $j\geq 1$; where $(P_1, Q_1) = (\mathbf{\bar g}_1, \mathbf{\bar g}_2)$ and $(P_2, Q_2)=(\mathbf{h}_1, \mathbf{h}_2)$. Applying Lemma~\ref{lifting complete level} yields the required result. 
\end{proof}

%% file: fibers.tex
\section{Fibers of cyclic elements in $\mathfrak{sl}_2(k)$ with respect to the Engel map}

In this section, we investigate the properties of fibers of the cyclic elements in $\mathfrak{sl}_2(k)$ with respect to the map $\bar e_m$. For any $g\in \mathfrak{sl}_2(k)$ and $m\geq 1$ we define the set:
$$E_{m}(g) = \left\{(x, y)\in \M_2(k)\times \M_2(k)\ \mid\ \bar e_m(x, y)=g\right\}.$$
The set $E_m(g)$ is the fiber set of $g$ corresponding to the $m$-th Engel map. We start with the case when $g$ is a cyclic nilpotent element in $\mathfrak{sl}_2(k)$. We assume that $char(k)\neq 2$ throughout the section.

\begin{proposition}{\label{nilpotentengel}}
Let $g=\left(\begin{array}{cc} 0 & 1 \\ 0 & 0 \end{array} \right) \in \M_2(k)$. Then there exists $(\alpha_1, \alpha_2)\in E_m(g)$ such that $\alpha_1$ is nilpotent and $\alpha_2$ is regular semisimple.
\end{proposition}
\begin{proof} Consider $h_1 = \left(\begin{array}{cc}
0 & 1 \\ 0 & 0 \end{array}\right)$ and $h_2 = \left(\begin{array}{cc} 1 & 0 \\  0 & -1 \end{array} \right)$ in $\M_2(k)$.  Now, we claim that 
$$\bar e_m(h_1,h_2)=\left(\begin{array}{cc} 0 & (-2)^m \\ 0 & 0 \end{array}\right)$$ for any $m\in \mathbb{N}$. We prove this by the following induction process:

\textbf{Base case :} It is easy to see that $\bar e_1(h_1, h_2) = h_1h_2 - h_2h_1 = \left(\begin{array}{cc} 0 & -2 \\
0 & 0 \end{array}\right)$. Hence, the claim is true for $m=1$.

\textbf{Induction Hypothesis :} Let the statement is true for $m=2,3,\ldots,r$.

\textbf{Inductive step :} By the induction hypothesis, we have 
$$\bar e_r(h_1,h_2)=\left(\begin{array}{cc} 0 & (-2)^r \\ 0 & 0 \end{array}\right).$$

Now, \begin{eqnarray*}
    \bar e_{r+1}(h_1, h_2) &=& [\bar e_r(h_1,h_2),h_2] =  \left(\begin{array}{cc} 0 & (-2)^r \\ 0 & 0 \end{array}\right) \left(\begin{array}{cc} 1 & 0 \\ 0 & -1 \end{array}\right) - \left(\begin{array}{cc} 1 & 0 \\ 0 & -1 \end{array}\right) \left(\begin{array}{cc}
0 & (-2)^r \\ 0 & 0 \end{array}\right)\\
 &=&   \left(\begin{array}{cc} 0 & (-2)^{r+1} \\ 0 & 0 \end{array}\right).
\end{eqnarray*} 
Therefore, the claim is true for any $m\in \mathbb{N}$.

It is easy to observe that there exists $P\in \GL_2(k)$ such that $P\bar e_m(h_1, h_2) P^{-1} = \left(\begin{array}{cc} 0 & 1 \\ 0 & 0 \end{array} \right)=g$. Consider $\alpha_i=Ph_iP^{-1}$ for $i=1,2$ and hence $\bar e_m(\alpha_1,\alpha_2)=g$. 

     Observe that $\alpha_1^2=\mathbf{0}$ as $h_1^2=\mathbf{0}$, and $h_2$ has distinct eigenvalues, hence $\alpha_2$ is regular semisimple.
\end{proof}
\begin{remark}
The elements obtained in the proof above show that there exist $\alpha_1$ and $\alpha_2$ in  $\mathfrak{sl}_2(k)$ such that $\bar e_m(\alpha_1, \alpha_2) = \left(\begin{array}{cc} 0 & 1 \\ 0 & 0
\end{array} \right)$.
\end{remark}

\begin{lemma}{\label{fiber 1}}
Let $g\in \mathfrak{sl}_2(k)$ be a cyclic element. Then, for any $(h_1, h_2)\in E_1(g)$, $Z(h_1)\cap Z(h_2)=\{\mathbf{0}\}$.
\end{lemma}
\begin{proof}
Let, $g\in \mathfrak{sl}_2(k)$ be cyclic and $h_1, h_2\in \M_2(k)$ be such that $\bar e_1(h_1, h_2)=g$. Let $Q\in Z(h_1)\cap Z(h_2)$ in $\mathfrak{sl}_2(k)$. As $g$ is cyclic, $h_1, h_2$ must be non-scalar, hence cyclic. Now, $Q\in Z(h_1)$ implies $Q= f(h_1)$ for some $f(t)\in k[t]$. By division algorithm, there exist $s(t)$ and $r(t)$ in $k[t]$ such that $$ f(t)=s(t)\chi_{h_1}(t) + r(t)$$ where either $r(t)=0$ or $\deg(r(t)) < \deg(\chi_{h_1}(t)) = 2$. 
    
Let $\deg(r(t)) = 1$ and $r(t) = at + b$ for some $a, b\in k$. Thus, $Q= f(h_1) = ah_1 + bI$. Now, $Q\in Z(h_2)$ implies $Qh_2=h_2Q$. Hence, we obtain
$Qh_2= (ah_1 + bI)h_2 = h_2Q = h_2(ah_1 + bI)$ which gives $a[h_1,h_2]=\mathbf{0}$. 
Hence, $a g = \mathbf{0}$ which means $a=0$ as $g\neq \mathbf{0}$. Substituting $a=0$ we get $Q=bI$. Now, $\tr(Q)=0$ thus $b=0$, consequently $Q=\mathbf{0}$. 
In the case, $r(t)=0$, clearly $Q=\mathbf{0}$. This proves the lemma.
\end{proof}
\begin{lemma}{\label{fiberortho}}
Let $h\in M_2(k)$ be any regular semisimple element. Then $Z(h)\cap Z(h)^{\perp}=\{\mathbf{0}\}$ in $\mathfrak{sl}_2(k)$.
\end{lemma}
\begin{proof}
Let $Q\in Z(h)\cap Z(h)^{\perp}$ in $\M_2(k)$. As $h$ is regular semisimple, there exists $\Lambda\in \GL_2(\bar k)$ such that $\Lambda h \Lambda^{-1} = \widehat{h} = \left(\begin{array}{cc} c & 0 \\ 0 & d
\end{array}\right)$ for some $c, d \in \bar k$ with $c\neq d$. Therefore, $Z(\widehat{h}) = span \left(\begin{array}{cc} 1 & 0 \\ 0 & -1 \end{array} \right)$. It is easy to see, that $Z(\widehat{h})^{\perp} = span\left \{\left(\begin{array}{cc} 0 & 1 \\ 0 & 0 \end{array} \right), \left(\begin{array}{cc} 0 & 0 \\ 1 & 0 \end{array} \right) \right\}$. Hence, $Z(\widehat{h})\cap Z(\widehat{h})^{\perp}=\mathbf{0}$ in $\mathfrak{sl}_2(\bar k)$.

Now $\Lambda Q\Lambda^{-1}\in Z(\widehat{h})$. Moreover, by the conjugacy invariance of the bilinear form (over $\bar k)$ we obtain $\langle Q, h\rangle = 0$ implies $\langle \Lambda Q \Lambda^{-1}, \widehat{h} \rangle = 0$. Hence $\Lambda Q \Lambda^{-1} \in Z(\widehat{h}) \cap Z(\widehat{h})^{\perp}$. Hence, $\Lambda Q \Lambda^{-1} = \mathbf{0}$. This further implies $Q=\mathbf{0}$. Hence the lemma.
\end{proof}

\begin{proposition}{\label{cyclicnilpotentcentral}}
Let $\mathbf{g}$ be a cyclic but not regular semisimple element in $\mathfrak{sl}_2(k)$. Then any lift $A$ of $\mathbf{g}$ in $\mathfrak{sl}_2^{\circ}(\mathcal{O})$ is in $\im(e_{m+1}^{\mathcal{O}})$ for any $m\geq 1$.
\end{proposition}
\begin{proof}
As $g$ is cyclic but not regular semisimple in $\mathfrak{sl}_2(k)$, therefore $g$ is $k$-conjugate to $\left(\begin{array}{cc} 0 & 1 \\ 0 & 0 \end{array} \right)$. Now, by proposition~\ref{nilpotentengel} there exist $\beta_1,\beta_2\in \M_2(k)$ such that $$\left(\begin{array}{cc} 0 & 1 \\ 0 & 0  \end{array} \right) = \bar e_{m+1}(\beta_1, \beta_2) = [\bar e_m(\beta_1, \beta_2), \beta_2]$$ where $\beta_2$ is regular semisimple. However, $\left(\begin{array}{cc}
0 & 1 \\ 0 & 0 \end{array}\right)$ is cyclic therefore Lemma~\ref{fiber 1} implies that $Z(\bar e_m(\beta_1, \beta_2)) \cap Z(\beta_2) = \{\mathbf{0}\}$. Moreover, Lemma~\ref{fiberortho} ensures that $Z(\beta_2)\cap Z(\beta_2)^{\perp} = \{\mathbf{0}\}$ in $\mathfrak{sl}_2(k)$. Hence, invoking Proposition~\ref{adelic 2}, the result follows.
\end{proof}
\noindent We now restrict our attention to the case where $g$ is a regular semisimple element in $\mathfrak{sl}_2(k)$.

\begin{proposition}{\label{lem h2 regular semisimp}}
Let $g\in \mathfrak{sl}_2(k)$ be a regular semisimple element such that $g = \bar e_{m+1}(h_1, h_2)$ for some $h_1, h_2 \in \M_2(k)$ where $m\geq 1$. Then, $h_2$ must be a regular semisimple element.
 \end{proposition}
 \begin{proof}
As $g$ is regular semisimple, $h_1, h_2$ cannot be scalar matrices. On the contrary, suppose $h_2$ is not regular semisimple. Then there exist $J\in \GL_2(k)$ such that $Jh_2J^{-1} = \left(\begin{array}{cc} \lambda & 1 \\ 0 & \lambda \end{array} \right) = \beta_2$, say, for some $\lambda\in k$. Let $\beta_1 = Jh_1J^{-1}$, $\widehat{g} = JgJ^{-1}$ and write $\beta_2=\lambda I + N_1$ where $N_1=\left(\begin{array}{cc} 0 & 1 \\ 0 & 0    \end{array}\right)$. Therefore $\widehat{g} = \bar e_{m+1}(\beta_1, \beta_2)$.

Now we claim that $$\bar e_{m+1}(\beta_1,\beta_2)=\bar e_{m+1}(\beta_1,N_1)$$ for any $m\in\mathbb{N}$. We prove this claim by the following induction process:

\textbf{Base case :} We observe that $[\beta_1, \beta_2] = [\beta_1, \lambda I + N_1] = [\beta_1, N_1] =\bar e_1(\beta_1,N_1)$. 
Similarly, 
$$\bar e_2(\beta_1,\beta_2) = [\bar e_1(\beta_1,\beta_2),\beta_2] = [\bar e_1(\beta_1, N_1),\lambda I+N_1]= [\bar e_1(\beta_1,N_1),N_1]=\bar e_2(\beta_1,N_1).$$
Hence, the claim is true for $m=1$.

\textbf{Induction Hypothesis :} Let the claim is true for $m=2,3,...,r$.

\textbf{Inductive step :} By the induction hypothesis we have $$\bar e_{r+1}(\beta_1,\beta_2)=\bar e_{r+1}(\beta_1,N_1).$$

Now \begin{eqnarray*}
    \bar e_{r+2}(\beta_1,\beta_2)=[\bar e_{r+1}(\beta_1,\beta_2),\beta_2] = [\bar e_{r+1}(\beta_1, N_1),\lambda I+N_1]= [\bar e_{r+1}(\beta_1,N_1),N_1]=\bar e_{r+2}(\beta_1,N_1).
\end{eqnarray*}
Therefore, the statement is true for all $m\in \mathbb{N}$.

Now $\mathfrak{sl}_2(k) = span\left\{ \Delta = \left(\begin{array}{cc} 1 & 0 \\ 0 & -1   \end{array}\right), \ N_1=\left(\begin{array}{cc} 0 & 1 \\ 0 & 0 \end{array}\right), \ N_2 = \left(\begin{array}{cc} 0 & 0 \\ 1 & 0  \end{array} \right) \right\}$. Consider $\widehat \beta_1=\beta_1-2^{-1}\tr(\beta_1)I$. Then it is easy to see that $\bar e_1(\widehat{\beta}_1,N_1)=\bar e_1(\beta_1,N_1)$ and hence we obtain
\begin{equation}{\label{lie 1}}
\bar e_{m+1}(\beta_1,N_1)=\bar e_{m+1}(\widehat{\beta}_1,N_1);
\end{equation}
for $m\geq 1$ where $\widehat{\beta}_1 \in \mathfrak{sl}_2(k)$. Therefore,
\begin{equation}{\label{lie 2}}
\widehat \beta_1 = \gamma_1 \Delta + \gamma_2 N_1 + \gamma_3 N_2\ \ \  \text{for some}\  \gamma_1, \gamma_2, \gamma_3 \in k.
\end{equation}
Note that $\bar e_1(\Delta,N_1) = \Delta N_1-N_1\Delta=2N_1$ which implies $\bar e_2(\Delta, N_1) = \mathbf{0}$. Moreover, $\bar e_1(N_1, N_1) = \mathbf{0}$ and $\bar e_1(N_2, N_1)= N_2N_1-N_1N_2= -\Delta$. This gives, $\bar e_2(N_2,N_1)=\bar e_1(-\Delta,N_1)=-2N_1$ and hence $\bar e_3(N_2,N_1)=\mathbf{0}$. By Equation~(\ref{lie 1}) and~(\ref{lie 2}) we obtain \begin{equation}{\label{lie 3}}
\bar e_{m+1}(\beta_1,N_1)=\bar e_{m+1}(\widehat{\beta}_1,N_1)=\gamma_1\bar e_{m+1}(\Delta,N_1)+\gamma_2\bar e_{m+1}(N_1,N_1)+\gamma_3\bar e_{m+1}(N_2,N_1).
\end{equation}

 Therefore, we have the following two cases:
 \begin{enumerate}
\item When $m\geq 2$, by Equation~(\ref{lie 3}) we obtain $\bar e_{m+1}(\beta_1,N_1)=\mathbf{0}$. Hence $\bar e_{m+1}(\beta_1,\beta_2)=\mathbf{0}$.
\item When $m=1$, by Equation~(\ref{lie 3}) we obtain $\bar e_{m+1}(\beta_1, N_1) = -2\gamma_3N_1$. Hence $\bar e_{m+1}(\beta_1, \beta_2) = -2\gamma_3N_1$.
\end{enumerate}
Therefore, for $m\geq 1$, $\bar e_{m+1}(\beta_1, \beta_2)$ is nilpotent. As $\widehat{g}=\bar e_{m+1}(\beta_1, \beta_2)$ therefore $\widehat{g}$ must be nilpotent, which is a contradiction (as $g$ is regular semisimple by assumption). Hence $h_2$ must be regular semisimple.
\end{proof}
\begin{remark}{\label{regular semi criterion lie algebra}}
It is easy to see from the proof of Proposition~\ref{lem h2 regular semisimp} that, if $A=\bar e_{m+1}(g_1, g_2)$ for some $g_1, g_2\in \M_2(k)$ then there exist $\widehat{g}_1=g_1-2^{-1}\tr(g_1)I$ in $\mathfrak{sl}_2(k)$ such that $A=\bar e_{m+1}(\widehat{g}_1,g_2)$.
\end{remark}

\begin{proposition}{\label{regularsemisimplecentral}}
Let $\mathbf{g}$ be a regular semisimple element in $\mathfrak{sl}_2(k)$ and $m\geq 1$ be a positive integer. Then, any lift $A$ of $\mathbf{g}$ in $\mathfrak{sl}_2^{\circ}(\mathcal{O})$ belongs to $\im(e_{m+1}^{\mathcal{O}})$ if and only if $\mathbf{g}$ belongs to $\im(\bar e_{m+1})$. 
\end{proposition}
\begin{proof}
Let, $A\in \mathfrak{sl}_2^{\circ}(\mathcal{O})$. Then it is obvious that $\mathbf{g}\in \im(\bar e_{m+1})$. We need to prove the converse part.

Let $\mathbf{g}\in \im(\bar e_{m+1})$, where $m\geq 1$. Then, there exist $h_1, h_2\in \M_2(k)$ such that $\bar e_{m+1}(h_1, h_2) = \mathbf{g}$. By Proposition~\ref{lem h2 regular semisimp}, we get that $h_2$ must be regular semisimple. Therefore, Lemma~\ref{fiberortho} implies $Z(h_2)\cap Z(h_2)^{\perp} = \{\mathbf{0}\}$ in $\mathfrak{sl}_2(k)$. As $\mathbf{g} = [\bar e_m(h_1, h_2), h_2]$, by Lemma~\ref{fiber 1} we have $Z(\bar e_m(h_1, h_2)) \cap Z(h_2) = \{\mathbf{0}\}$. Invoking Proposition~\ref{adelic 2}, our result follows.
 \end{proof}

%% file: surj.tex
\section{Surjectivity results}
In this section, we prove our main theorems.

\begin{proposition}{\label{liftingcentralresult1}}
Let $\mathcal{O}$ be a local principal ideal ring, complete with respect to its maximal ideal $\mathfrak{m}=(\pi)$ with residue field $k$ of characteristic $\neq 2$. Let $A$ be any element of $\mathfrak{sl}_2^{\circ}(\mathcal{O})$ such that $\theta(A)=:\bar A\in \mathfrak{sl}_2(k)$ is non-zero and $m\geq 1$ be a positive integer. Then, $A\in \im (e_{m+1}^{\mathcal{O}})$ in $\M_2(\mathcal{O})$ if and only if $\bar A\in \im(\bar e_{m+1})$ in $\M_2(k)$.
\end{proposition}
\begin{proof}
As $\bar A\in \mathfrak{sl}_2(k)$ is non-trivial, either $\bar A$ is regular semisimple or $\bar A$ is nilpotent. Let $\bar A\in \im(\bar e_{m+1})$ in $\M_2(k)$. Then we have the following cases:
\begin{enumerate}
\item When $\bar A$ is regular semisimple in $\mathfrak{sl}_2(k)$: In this case Proposition~\ref{regularsemisimplecentral} ensures that $A\in \im(e_{m+1}^{\mathcal{O}})$ in $\M_2(\mathcal{O})$.
\item  When $\bar A$ is non-trivial nilpotent element in $\mathfrak{sl}_2(k)$: The Proposition~\ref{cyclicnilpotentcentral} implies  $A\in \im(e_{m+1}^{\mathcal{O}})$ in $\M_2(\mathcal{O})$.
\end{enumerate}
Hence, in both the cases $A\in \im(e_{m+1}^{\mathcal{O}})$ in $\M_2(\mathcal{O})$. On the other hand, when $A\in \im(e_{m+1}^{\mathcal{O}}) $ in $\M_2(\mathcal{O})$, it clearly implies $\bar A\in \im(\bar e_{m+1})$ in $\M_2(k)$. This completes the proof.
\end{proof}

\begin{lemma}{\label{splitting scalar and trace zero}}
For any $\ell\geq 1$, every matrix $g\in \M_2(\mathcal{O}_{\ell})$ can be written as $g=\alpha + \beta$ where $\alpha$ is a scalar matrix in $\M_2(\mathcal{O}_{\ell})$ and $\beta\in \mathfrak{sl}_2^{\circ}(\mathcal{O}_{\ell})$. 
\end{lemma}
\begin{proof}
Let, $g=\left(\begin{array}{cc} a & b \\ c & d \end{array} \right)$ be a matrix in $\M_2(\mathcal{O}_{\ell})$. Consider $\alpha = \left(\begin{array}{cc} \frac{a+d}{2} & 0 \\ 0 & \frac{a+d}{2} \end{array}\right)$ and $\beta = \left(\begin{array}{cc} \frac{a-d}{2} & b \\ c & \frac{d-a}{2} \end{array}\right)$. Clearly, $\alpha$ is a scalar matrix and $\tr(\beta)=0$. Hence, the result follows.
\end{proof}
Next, we deal with the lifts of zero.
\begin{lemma}{\label{lift of zero}}
Suppose $\im(\bar e_{m+1})=\mathfrak{sl}_2(k)$ and $A$ is a non-zero lift of $\mathbf{0}$ in $\mathfrak{sl}_2^{\circ}(\mathcal{O})$. Then, for $m\geq 1$, $A\in \im(e_{m+1}^{\mathcal{O}})$.
\end{lemma}
\begin{proof}
Let $A\neq \mathbf{0}$ in $\mathfrak{sl}_2^{\circ}(\mathcal{O})$, which is a lift of $\mathbf{0}$. Then, there exists an $\ell\geq 2$ such that $A_{\ell}\neq \mathbf{0}$ in $\M_2(\mathcal{O}_{\ell})$ and $A_{\ell-1}=\mathbf{0}$ in $\M_2(\mathcal{O}_{\ell-1})$. As $ker(\theta_{\ell-1})$ is generated by $\pi_{\ell}^{\ell-1}$ and $A_{\ell}\in \mathfrak{sl}_2^{\circ}(\mathcal{O}_{\ell})$, therefore by Lemma~\ref{splitting scalar and trace zero} we can write $A_{\ell}=\pi^{\ell-1}_{\ell}\left(\begin{array}{cc} \delta_{\ell} & u_{\ell} \\ v_{\ell} & -\delta_{\ell} \end{array}\right)$ for some $\delta_{\ell}, u_{\ell}, v_{\ell} \in \mathcal{O}_{\ell}$ where at least one of the $\delta_{\ell}, u_{\ell}, v_{\ell} \in \mathcal{O}_{\ell}$ is an unit in $\mathcal{O}_{\ell}$. Note that $A_{\ell}$ has trace-zero, hence $\pi_{\ell}^{\ell-1}$ times the scalar part in the decomposition given by Lemma~\ref{splitting scalar and trace zero} will be zero. As $A_{\ell+1}$ is a lift of $A_{\ell}$ therefore 
$$A_{\ell+1} = \pi_{\ell+1}^{\ell-1}\left(\begin{array}{cc} \delta_{\ell+1} & u_{\ell+1} \\ v_{\ell+1} & -\delta_{\ell+1} \end{array}\right) + \pi_{\ell+1}^{\ell}C \ \ \text{for some}\ \  C\in\M_2(\mathcal{O}_{\ell+1}),$$  
where $\theta_{\ell}\left(\pi_{\ell+1}^{\ell-1} \left(\begin{array}{cc} \delta_{\ell+1} & u_{\ell+1} \\
v_{\ell+1} & -\delta_{\ell+1} \end{array}\right) \right) = A_{\ell}$ and $\delta_{\ell+1}, u_{\ell+1}, v_{\ell+1}$ are the lifts of $\delta_{\ell}, u_{\ell}, v_{\ell}$ respectively with respect to the map $\theta_{\ell}$. 
      
Again, by using the decomposition in Lemma~\ref{splitting scalar and trace zero} and the fact that $\tr(A_{\ell+1}) = 0$, we obtain $\pi_{\ell+1}^{\ell}$ times the scalar part of $C$ will be zero. Hence, $\pi^{\ell}_{\ell+1}C = \pi_{\ell+1}^{\ell} \left(\begin{array}{cc} c_1 & c_2 \\ c_3 & -c_1 \end{array}\right)$ for some $c_1, c_2, c_3\in \mathcal{O}_{\ell+1}$. This further implies $A_{\ell+1} = \pi_{\ell+1}^{\ell-1}\left(\begin{array}{cc} \delta_{\ell+1}' & u_{\ell+1}' \\ v_{\ell+1}' & -\delta_{\ell+1}' \end{array}\right)$, where 
$$ \delta_{\ell+1}' = \delta_{\ell+1} + \pi_{\ell+1}c_1, \  u_{\ell+1}' = u_{\ell+1} + \pi_{\ell+1}c_2 \ \text{and}\ v_{\ell+1}' = v_{\ell+1} + \pi_{\ell+1}c_3.$$ 
It is easy to observe that 
$$\theta_{\ell}(\pi_{\ell+1}^{\ell-1} \delta_{\ell+1}') = \pi_{\ell}^{\ell-1}\delta_{\ell},\ \theta_{\ell}(\pi_{\ell+1}^{\ell-1}u_{\ell+1}') = \pi_{\ell}^{\ell-1}u_{\ell}\ \text{and}\  \theta_{\ell}(\pi_{\ell+1}^{\ell-1}v_{\ell+1}') = \pi_{\ell}^{\ell-1}v_{\ell}$$ 
and hence at least one of $\delta_{\ell+1}', u_{\ell+1}', v_{\ell+1}'$ must be an unit in $\mathcal{O}_{\ell+1}$.
   
Inductively using the above steps, and the canonical isomorphism $\mathcal{T}^*$ (described in Section~\ref{englifting}), we obtain that there exists a positive integer $\mu$ such that $A=\pi^{\mu} \left(\begin{array}{cc} \delta & u \\ v & -\delta \end{array} \right)$ for some $\delta, u, v\in \mathcal{O}$ and at least one of $\delta, u, v$ is a unit in $\mathcal{O}$. Therefore, if we assume $\mathfrak{X} = \left(\begin{array}{cc} \delta & u \\ v & -\delta \end{array}\right)$ then $\mathfrak{X}\in \mathfrak{sl}_2^{\circ}(\mathcal{O})$ and $\theta(\mathfrak{X}) = \mathfrak{\bar X}$ is non-trivial in $\mathfrak{sl}_2(k)$. As $\im(\bar e_{m+1}) = \mathfrak{sl}_2(k)$ by assumption, therefore $\mathfrak{\bar X}\in \im(\bar e_{m+1})$. Invoking Proposition~\ref{liftingcentralresult1} we obtain $\mathfrak{X}\in \im(e_{m+1}^\mathcal{O})$ in $\M_2(\mathcal{O})$. Therefore, there exist $\Lambda_1, \Lambda_2 \in \M_2(\mathcal{O})$ such that $e_{m+1}^{\mathcal{O}}(\Lambda_1, \Lambda_2) = \mathfrak{X}$. As $A = \pi^{\mu} \mathfrak{X}$ therefore $A = e_{m+1}^{\mathcal{O}}(\pi^{\mu}\Lambda_1, \Lambda_2)$. This completes the proof.
\end{proof}

Now we are ready to prove our first main theorem.
\begin{theorem}{\label{th1}}
Let $\mathcal{O}$ be a local principal ideal ring, complete with respect to its maximal ideal $\mathfrak{m} = (\pi)$ with residue field $k$ of characteristic $\neq 2$ and $m\geq 1$ be a positive integer. Then, $\im(e_{m+1}^{\mathcal{O}}) = \mathfrak{sl}_2^{\circ}(\mathcal{O})$ if and only if $\im(\bar e_{m+1}) = \mathfrak{sl}_2(k)$.
\end{theorem}
\begin{proof}
Using the natural surjection $\mathfrak{sl}_2^{\circ}(\mathcal{O}) \rightarrow \mathfrak{sl}_2(k)$, it is easy to observe that if $\im(e_{m+1}^{\mathcal{O}}) = \mathfrak{sl}_2^{\circ}(\mathcal{O})$ then $\im(\bar e_{m+1}) = \mathfrak{sl}_2(k)$. So, we are left to prove the converse part. 

Let $\im(\bar e_{m+1}) = \mathfrak{sl}_2(k)$ in $\M_2(k)$. Let $A\in \mathfrak{sl}_2^{\circ}(\mathcal{O})$ and $\bar A\in \mathfrak{sl}_2(k)$. We have the following cases:
\begin{enumerate}
\item When $\bar A$ is non-zero in $\mathfrak{sl}_2(k)$: In this case, Proposition~\ref{liftingcentralresult1} gives $A\in \im(e_{m+1}^{\mathcal{O}}) $.
\item When $\bar A=\mathbf{0}$ in $\mathfrak{sl}_2(k)$: In this case, there are two possibilities: (1) $A=\mathbf{0}$ and (2) $A\neq \mathbf{0}$.
\begin{enumerate}
\item When $A=\mathbf{0}$, clearly, $A\in \im(e_{m+1}^{\mathcal{O}}) $ in $\M_2(\mathcal{O})$.
\item When $A\neq \mathbf{0}$, the Lemma~\ref{lift of zero} gives that $A\in \im(e_{m+1}^{\mathcal{O}})$. 
\end{enumerate}
\end{enumerate}
Therefore, $A\in \im(e_{m+1}^{\mathcal{O}})$ in $\M_2(\mathcal{O})$. Hence $\mathfrak{sl}_2^{\circ}(\mathcal{O})\subset \im(e_{m+1}^{\mathcal{O}})$ in $\M_2(\mathcal{O})$. It is easy to see that $\im(e_{m+1}^{\mathcal{O}}) \subset \mathfrak{sl}_2^{\circ}(\mathcal{O})$. This proves that $\im(e_{m+1}^{\mathcal{O}})= \mathfrak{sl}_2^{\circ}(\mathcal{O})$. 
\end{proof}

Now, we need to deal with the surjectivity of these maps. For that, we require the following lemma.
\begin{lemma}{\label{from M2 to trace zero}}
Let $g\in \mathfrak{sl}_2^{\circ}(\mathcal{O}_{\ell})$ and $e_{m+1}^{\mathcal{O}_{\ell}}(h_1,h_2)=g$ for some $h_1, h_2\in \mathfrak{sl}_2^{\circ}(\mathcal{O}_{\ell})$. Suppose $A$ is a lift of $g$ in $\mathfrak{sl}_2^{\circ}(\mathcal{O}_{\ell+1})$ and $A = e_{m+1}^{\mathcal{O}_{\ell+1}}(s_1, s_2)$ for some $s_1, s_2\in \M_2(\mathcal{O}_{\ell+1})$ which are lifts of $h_1$ and $h_2$ respectively. Then, there exist $\widehat{s}_1, \widehat{s}_2\in \mathfrak{sl}_2^{\circ}(\mathcal{O}_{\ell+1})$ such that $e_{m+1}^{\mathcal{O}_{\ell+1}}(\widehat{s}_1,\widehat{s}_2) = A$, and $\widehat{s}_1, \widehat{s}_2$ are the lifts of $h_1$ and $h_2$ respectively.
\end{lemma}
\begin{proof}
Since $s_i$ is the lift of $h_i$, $\tr(s_i)\in ker(\theta_{\ell})$ for $i=1, 2$. Hence, $\tr(s_i) = \pi_{\ell+1}^{\ell}u_i$ for some $u_i\in \mathcal{O}_{\ell+1}$. Now, by Lemma~\ref{splitting scalar and trace zero}, $s_i=\widehat{s}_i+c_i$ where $\tr(\widehat{s}_i)=0$ and $c_i = \gamma_iI$ for some $\gamma_i\in \mathcal{O}_{\ell+1}$. Taking trace on both sides of this equation, we get  $\pi_{\ell+1}^{\ell} u_i = 2\gamma_i$ for $i=1, 2$. Therefore, $c_i = \pi_{\ell+1}^{\ell} \left(\begin{array}{cc} \frac{u_i}{2} &  \\  & \frac{u_i}{2} \end{array}\right)$. Consequently, \begin{eqnarray*}
A &=& e_{m+1}^{\mathcal{O}_{\ell+1}}(s_1,s_2) = e_{m+1}^{\mathcal{O}_{\ell+1}}(\widehat{s}_1, \widehat{s}_2) + \pi_{\ell+1}^{\ell} D e^{\mathcal{O}_{\ell+1}}_{m+1}\left(\frac{u_1}{2}I, \frac{u_2}{2}I\right) = e_{m+1}^{\mathcal{O}_{\ell+1}}(\widehat{s}_1,\widehat{s}_2)
\end{eqnarray*}
as $De^{\mathcal{O}_{\ell+1}}_{m+1}(2^{-1}u_1I,2^{-1}u_2I) =\mathbf{0}$ by Remark~\ref{scalarmatrixderivativemap}.
This gives the required result.
\end{proof}

Now, we prove our second main theorem of this article.
\begin{theorem}{\label{th2}}
Let $\mathcal{O}$ be a local principal ideal ring, complete with respect to its maximal ideal $\mathfrak{m}=(\pi)$ with residue field $k$ of characteristic $\neq 2$ and $m\geq 1$ be any given positive integer. Then, $e_{m+1}^{\mathcal{O}}\colon \mathfrak{sl}_2^{\circ}(\mathcal{O}) \times \mathfrak{sl}_2^{\circ}(\mathcal{O}) \rightarrow\mathfrak{sl}_2^{\circ}(\mathcal{O})$ is surjective if and only if $\bar e_{m+1}\colon \mathfrak{sl}_2(k) \times \mathfrak{sl}_2(k) \rightarrow \mathfrak{sl}_2(k)$ is surjective.
\end{theorem}
\begin{proof}
First, suppose $e_{m+1}^{\mathcal{O}}\colon \mathfrak{sl}_2^{\circ}(\mathcal{O}) \times \mathfrak{sl}_2^{\circ}(\mathcal{O}) \rightarrow \mathfrak{sl}_2^{\circ}(\mathcal{O})$ is surjective. Take any $g\in \mathfrak{sl}_2(k)$. The reduction map $\mathfrak{sl}_2^{\circ}(\mathcal{O}) \rightarrow \mathfrak{sl}_2(k)$ induced by $\theta$ is a surjection. Take a lift $h$ in $\mathfrak{sl}_2^{\circ}(\mathcal{O})$ of $g$ under this surjection. Now from the given assumption $h = e_{m+1}^{\mathcal{O}}(h_1, h_2)$ for some $h_1, h_2\in \mathfrak{sl}_2^{\circ}(\mathcal{O})$. This implies $g=\bar e_{m+1}(\bar h_1, \bar h_2)$ where $\bar h_1, \bar h_2\in \mathfrak{sl}_2(k)$. This shows $\bar e_{m+1}\colon \mathfrak{sl}_2(k) \times \mathfrak{sl}_2(k) \rightarrow \mathfrak{sl}_2(k)$ is surjective.
    
Now, we need to show the converse. Let $\bar e_{m+1} \colon \mathfrak{sl}_2(k) \times \mathfrak{sl}_2(k) \rightarrow \mathfrak{sl}_2(k)$ is surjective. Take $A \in \mathfrak{sl}_2^{\circ}(\mathcal{O})$. Then we have the following cases:

\textbf{Step I:} First, we deal with the case when $\bar A$ is non-zero and lift the solution to $\mathcal O_2$.
\begin{enumerate}
\item Let $\bar A$ be non-trivial nilpotent in $\mathfrak{sl}_2(k)$ and $\bar A= \bar e_{m+1}(h_1, h_2)$ for some $h_1, h_2 \in \mathfrak{sl}_2(k)$. Moreover, by Lemma~\ref{nilpotentengel}, we can choose $h_2$ a regular semisimple element in $\mathfrak{sl}_2(k)$. In this case, by Proposition   \ref{cyclicnilpotentcentral}, $A_2\in \im(e_{m+1}^{\mathcal{O}_2})$ in $\M_2(\mathcal{O}_2)$.
\item Let $\bar A$ be regular semisimple in $\mathfrak{sl}_2(k)$ and $\bar e_{m+1}(h_1, h_2) = \bar A$ for some $h_1, h_2\in \mathfrak{sl}_2(k)$. In this case, Proposition~\ref{lem h2 regular semisimp} ensures that $h_2$ must be regular semisimple. Therefore, by invoking Lemma~\ref{fiber 1}, Lemma~\ref{fiberortho} and Proposition~\ref{adelic 2}, we obtain $A_2\in \im(e_{m+1}^{\mathcal{O}_2})$ in $\M_2(\mathcal{O}_2)$. 
\end{enumerate}
Combining the above two cases we have, when $\bar A$ is a non trivial element then there exist $h_1, h_2\in \mathfrak{sl}_2(k)$ with $h_2$ a regular semisimple element such that $\bar e_{m+1}(h_1, h_2) = \bar A$.  By Proposition~\ref{adelic 2}, we get $A_2 = e_{m+1}^{\mathcal{O}_2}(\tilde h_1, \tilde h_2)$ for some $\tilde h_1, \tilde h_2\in \M_2(\mathcal{O}_2)$ where $\theta(\tilde h_i) = h_i$ for $i=1,2$. Further using Lemma~\ref{from M2 to trace zero} we get, there exist $h_1^{(2)}, h_2^{(2)}\in \mathfrak{sl}_2^{\circ}(\mathcal{O}_2)$ such that $e_{m+1}^{\mathcal{O}_2}(h_1^{(2)}, h_2^{(2)}) = A_2$ where $\theta(h_1^{(2)})=h_1, \theta(h_2^{(2)})=h_2$.

\textbf{Step II:} Now, we deal with the case when $\bar A$ is non-zero and lift the solution to $\mathcal O$. Consider $h_1', h_2'\in \M_2(\mathcal{O}_3)$ such that $\theta_2(h_i')=h_i^{(2)}$ for $i=1,2$. Then there exist $\mathfrak{D}_3 \in \mathfrak{sl}_2(\mathcal{O}_3)$ such that $A_3 = e_{m+1}^{\mathcal{O}_3}(h_1', h_2') + \delta_2 \mathfrak{D}_3$ where $\delta_2 = \pi_3^2$ is the generator of the $ker(\theta_2\colon  \mathcal{O}_3 \rightarrow \mathcal{O}_2)$. Note that $h_2'$ is regular semisimple, as $\theta(h_2') = h_2$, which is regular semisimple. By step III in the proof of Proposition~\ref{adelic 2}, we get $De_{m+1}^{\mathcal{O}_3}$ at $(h_1', h_2')$ is surjective onto $\mathfrak{sl}_2(\mathcal{O}_3)$. Consequently, there exist $\beta_1, \beta_2$ in $\mathfrak{sl}_2(\mathcal{O}_3)$ such that $De_{m+1}^{\mathcal{O}_3}(\beta_1, \beta_2) = \mathfrak{D}_3$. This implies $A_3 = e_{m+1}^{\mathcal{O}_3}(\tilde s_1, \tilde s_2)$ where $\tilde s_i = h_i' + \delta_2\beta_i$ for $i=1,2$. Note that $\tilde s_1, \tilde s_2 \in \M_2(\mathcal{O}_3)$ and $\theta_2(\tilde s_1) = h_1^{(2)}$, $\theta_2(\tilde s_2) = h_2^{(2)}$. Therefore, using Lemma~\ref{from M2 to trace zero} we obtain that there exist $h_1^{(3)}, h_2^{(3)}\in \mathfrak{sl}_2^{\circ}(\mathcal{O}_3)$ such that $A_3 = e_{m+1}^{\mathcal{O}_3}(h_1^{(3)}, h_2^{(3)})$, where $\theta_2(h_i^{(3)}) = h_i^{(2)}$ for $i=1,2$.

Continuing this process repeatedly and using step III in the proof of the Proposition~\ref{adelic 2}, Lemma~\ref{from M2 to trace zero}, and Lemma~\ref{lifting complete level}, we get the following. There exist $\mathbf{h}_1, \mathbf{h_2}\in \M_2(\mathcal{O})$ such that $A = e_{m+1}^{\mathcal{O}}(\mathbf{h}_1, \mathbf{h}_2)$, where $\mathcal{T}^*(\mathbf{h}_i)=(h_i, h_i^{(2)}, h_i^{(3)}, \ldots)$ in $\varprojlim\limits _{j\geq 1} \M_n(\mathcal{O}_j)$ for $i=1,2$. Under the isomorphism $\mathcal{O} \rightarrow \varprojlim\limits _{j\geq 1} \mathcal{O}_j$, the element $\tr(\mathbf{h}_i)$ maps to $(\tr(h_i), \tr(h_i^{(2)}), \tr(h_i^{(3)}), \ldots)$ for $i=1,2$. Since $(\tr(h_i), \tr(h_i^{(2)}), \tr(h_i^{(3)}), \ldots) = (0, 0, \ldots)$ therefore $\tr(\mathbf{h}_i) = 0$ for $i=1, 2$. This gives $\mathbf{h}_1, \mathbf{h}_2\in \mathfrak{sl}_2^{\circ}(\mathcal{O})$.

\textbf{Step III:}
We are left with the case when $\bar A = \mathbf{0}$. Again, in this case, there are two possibilities: (1) $A=\mathbf{0}$ and (2) $A\neq \mathbf{0}$.
\begin{enumerate}
\item When $A=\mathbf{0}$, it is clear that $A=e_{m+1}^{\mathcal{O}}(\mathbf{0}, \mathbf{0}) $ in $\mathfrak{sl}_2^{\circ}(\mathcal{O})$.
\item When $A\neq \mathbf{0}$ there exist a positive interger $\mu$ such that $A=\pi^{\mu} \mathfrak{X}$ where $\theta(\mathfrak{X}) = \mathfrak{\bar X}\neq \mathbf{0}$ and $\mathfrak{X}\in \mathfrak{sl}_2^{\circ}(\mathcal{O})$ (see  the proof of Lemma~\ref{lift of zero}). As $\mathfrak{\bar X}$ is non-trivial, therefore by step I and step II above we obtain that there exist $\mathbf{g}_1, \mathbf{g}_2\in \mathfrak{sl}_2^{\circ}(\mathcal{O})$ such that $\mathfrak{X} = e_{m+1}^{\mathcal{O}}(\mathbf{g}_1, \mathbf{g}_2)$. As $A=\pi^{\mu} \mathfrak{X}$ therefore $A = e_{m+1}^{\mathcal{O}}(\pi^{\mu}\mathbf{g}_1, \mathbf{g}_2)$. Note that $\pi^{\mu}\mathbf{g}_1\in \mathfrak{sl}_2^{\circ}(\mathcal{O})$ as $\tr(\mathbf{g}_1)=0$. 
\end{enumerate}
Therefore, $A = e_{m+1}^{\mathcal{O}}(P_1, P_2)$ for some $P_1, P_2\in \mathfrak{sl}_2^{\circ}(\mathcal{O})$. hence, the map $e_{m+1}^{\mathcal{O}}\colon \mathfrak{sl}_2^{\circ}(\mathcal{O}) \times \mathfrak{sl}_2^{\circ}(\mathcal{O})\rightarrow\mathfrak{sl}_2^{\circ}(\mathcal{O})$ is surjective for $m\geq 1$. This completes the proof.
\end{proof}
\begin{corollary}{\label{largeenough}}
Let $\mathcal{O}$ be a local principal ideal ring, complete with respect to its maximal ideal $\mathfrak{m}=(\pi)$ with residue field $k$ of characteristic $\neq 2$ and $m\geq 1$ be a positive integer. Suppose $k$ is large enough so that $\bar e_{m+1}\colon \mathfrak{sl}_2(k) \times \mathfrak{sl}_2(k) \rightarrow \mathfrak{sl}_2(k)$ is surjective (this is ensured by \cite[corollary 4.4]{bandmanliealgebra}). Then, (using Theorem~\ref{th2} for these $k$) the $(m+1)$-th Engel map $e_{m+1}^{\mathcal{O}}\colon \mathfrak{sl}_2^{\circ}(\mathcal{O}) \times \mathfrak{sl}_2^{\circ}(\mathcal{O}) \rightarrow \mathfrak{sl}_2^{\circ}(\mathcal{O})$ is surjective.
\end{corollary}